\documentclass[11pt]{article}
\usepackage{amsmath,amssymb,amsthm,textcomp}
\usepackage{amssymb,amsmath, amsfonts}
\usepackage{amsfonts,graphicx}
\usepackage{enumerate}
\usepackage{amsthm}
\usepackage{epsfig}
\usepackage{xcolor}
\usepackage{graphicx}
\usepackage{caption}
\usepackage{subcaption}
\usepackage[title]{appendix}
\usepackage{url}
\usepackage{float}
\usepackage{enumitem}
\usepackage{multirow}
\floatstyle{plaintop}

\usepackage{mathtools}
\usepackage{microtype}
\usepackage{autobreak}
\usepackage{stackengine}
\usepackage{environ}
\usepackage{lscape}
\usepackage{biblatex}
\newcommand\numberthis{\addtocounter{equation}{1}\tag{\theequation}} 

\NewEnviron{myequation}{%
	\begin{align*}
		\scalebox{1.5}{$\BODY$}
	\end{align*}
}
\stackMath

\allowdisplaybreaks
\allowdisplaybreaks
\def\@biblabel#1{}

\makeatletter
\def\@biblabel#1{}
\makeatother 

\numberwithin{equation}{section}
\newtheorem{theorem}{Theorem}[]

\newtheorem{remark}{Remark}[]

\begin{document}
	\begin{center}
		{\Large\bf Record-Based Transmuted Generalized Linear Exponential Distribution with Increasing, Decreasing and Bathtub Shaped Failure Rates}  \\
		\vspace{0.3in}
				{\bf Mohd. Arshad$^{a,b}$, Mukti Khetan$^{c}$, \bf Vijay Kumar$^{d,}$\footnote{Corresponding author. E-mail addresses: ~arshad.iitk@gmail.com, arshad@iiti.ac.in (Mohd. Arshad), \\   mukti.khetan11@gmail.com (Mukti Khetan), dvijay.iitb@gmail.com
						(Vijay Kumar), ashokiitb09@gmail.com (Ashok Kumar Pathak).} and Ashok Kumar Pathak$^{e}$}
				\\
		$^{a}$Department of Mathematics,  Indian Institute of Technology Indore, Simrol, Indore, India. \\
			$^{b}$Department of Statistics and Operations Research, 
				Aligarh Muslim University, Aligarh, India.  \\
			$^{c}$Department of Mathematics, Amity University Mumbai, Maharashtra, India. \\
		$^{d}$Department of Mathematics, Indian Institute of Technology Bombay, Mumbai, India. \\		
		 $^{e}$Department of Mathematics and Statistics, Central University of Punjab, Bathinda, India.
	\end{center}
	\begin{abstract}
	The linear exponential distribution is a generalization of the exponential and Rayleigh distributions. This distribution is one of the best models to fit data with increasing failure rate (IFR). But it does not provide a reasonable fit for modeling data with decreasing failure rate (DFR) and bathtub shaped failure rate (BTFR). To overcome this drawback,  we propose a new record-based transmuted generalized linear exponential (RTGLE) distribution by using the technique of Balakrishnan and He (2021). The family of RTGLE distributions is more flexible to fit the data sets with IFR, DFR, and BTFR, and also generalizes several well-known models as well as some new record-based transmuted models. This paper aims to study the statistical properties of RTGLE distribution, like, the shape of the probability density function and hazard function, quantile function and its applications, moments and its generating function, order and record statistics, R\'enyi entropy. The maximum likelihood estimators, least squares and weighted least squares estimators, Anderson-Darling estimators, Cram\'er-von Mises estimators of the unknown parameters are constructed and their biases and mean squared errors are reported via Monte Carlo simulation study. Finally, the real data set based on failure time illustrates the goodness of fit and applicability of the proposed distribution; hence, suitable recommendations are forwarded.  
%
%
	\end{abstract}
	
	\vskip 2mm
	\section{Introduction}\label{sec1}
	The linear exponential (LE) distribution is a generalization of the exponential and Rayleigh distributions. It is a well-known probability distribution for modeling lifetime data in reliability and related experimental areas. Several models with monotonically increasing failure rate (IFR) and decreasing failure rate (DFR) have been studied in the reliability and survival analysis. The LE distribution is one of the best models to fit data with IFR. But it does not provide a reasonable fit for modeling data with DFR and bathtub shaped failure rate (BTFR).  To overcome this drawback, Mahmoud and Alam (2010) developed the generalized linear exponential distribution (GLE) for modeling the data with IFR, DFR and BTFR. The GLE distribution includes the exponential, Rayleigh, Weibull, linear exponential, and generalized Rayleigh distributions as sub-models. Several generalizations of the exponential and Weibull distributions have been developed in the literature to provide more flexible distributions for modeling of complex data sets. Several researchers constructed new probability models by using different techniques of construction. Thousands of new probability models have been developed and studied in the last two decades. Still, there are some requirements to develop a more flexible probability model than the existing classical one. Some of the more popular techniques of construction of a new probability model were studied by Kumaraswamy (1980), Marshall and Olkin (1997, 2007), Gupta and Kundu (1999), Eugene et al. (2002), Arnold et al. (2006), Ferreira and Steel (2006, 2007), Shaw and Buckley (2007, 2009), Jones (2009), Zografos and Balakrishnan (2009), Cordeiro and De Castro (2011), Singla et al. (2012), Mahmoudi and Jafari (2012, 2017), Almalki and Yuan (2013), Almalki and Nadarajah (2014), Azzalini (2013), Sarhan et al. (2013), Balakrishnan and Risti\'c (2016), He et al. (2016), Tahir and Cordeiro (2016), Zeng et al. (2016), Al-Babtain et al. (2017), Granzotto et al. (2017), Karakaya et al. (2017), Mahdavi and Kundu (2017), Ahmad and Ghazal (2020), Peña-Ramírez et al. (2020), Maurya and Nadarajah (2020), Tahir et al. (2020), Shakhatreh et al. (2021), Ghosh et al. (2021).
	
	Shaw and Buckley (2009) introduced a quadratic rank transmutation map to generate a new family of distributions. This technique is a powerful tool to generate the skewed probability distributions and is used in more than 400 research articles in recent past. The distribution function (DF) of the transmuted distribution can be obtained by 
	\begin{equation}\label{eq1}
		F_{\lambda}(x)=(1+\lambda) G(x)-\lambda G^2(x), \; x \in \mathbb{R},
	\end{equation}
	where $G(\cdot)$ denotes the DF of the baseline distribution and $-1 \leq \lambda \leq 1$. The transmuted DF, given in (\ref{eq1}), can be obtained by the mixture of the DFs of the smallest and the largest order statistics from the baseline distribution in a sample of size $2$. Kozubowski and Podgórski (2016) shown that the transmuted distributions are a special case of extremal distributions. Granzotto et al. (2017) extended the quadratic rank transmuted distribution to the cubic rank transmuted distribution by considering the mixture of the DF'S of the three order statistics from the baseline distribution. The rank cubic transmuted distribution can be obtained by  
	\begin{equation*}
		F_{\lambda_1,\lambda_2}(x)=\lambda_1 G(x)+ (\lambda_2-\lambda_1) G^2(x)+(1-\lambda_2) G^3(x), \; x \in \mathbb{R},
	\end{equation*}
	where $0\leq \lambda_1 \leq 1$, $-1 \leq \lambda_2 \leq 1$, and $G(\cdot)$ denotes the DF of the baseline distribution. A large number of papers devoted to elementary properties of various new quadratic and cubic transmuted distributions in the literature, see Aryal and Tsokos (2011), Elbatal and Aryal (2013),  Merovci (2013), Khan and King (2014), Tian et al. (2014), Nofal et al. (2017), Rahman et al. (2020), Al-Babtain et al. (2020).
	
	Recently, Balakrishnan and He (2021) proposed a record-based transmuted map to generate new probability models. The record-based transmuted distributions can be obtained by 
	\begin{equation}\label{eq3}
		F_{p}(x)=G(x)+p\bar{G}(x)\log\bar{G}(x), \; x \in \mathbb{R},
	\end{equation}
	where $0\leq p \leq 1$, and $G(\cdot)$ and $\bar{G}(\cdot)$ denote the DF and the survival function (SF) of the baseline distribution, respectively. The authors introduced a few new record-based transmuted (RT) probability distributions like RT-exponential (RTE) distribution, RT-Linear exponential (RTLE) distribution, RT-Weibull (RTW) distribution, etc. Tani\c{s} and Sarco\u{g}lu (2020) studied the properties and inferential issues of the RTW distribution.
	
	 This paper aims to propose a new generalization of the GLE distribution by using the technique of Balakrishnan and He (2021) and study some statistical properties. The new generalization provides RT-generalized linear exponential (RTGLE) distribution, which is a more general class of probability distributions including the RTE distribution, RTLE distribution, RTW distribution, Weibull distribution, etc. (see Table \ref{table1} for subclass of distributions). RTGLE distribution is flexible to fit the data sets with IFR, DFR and BTFR. The article's structure is as follows: Section \ref{sec2} discusses the genesis of the record-based transmuted approach. Section \ref{sec3} provides the DF, probability density function (PDF), and hazard function (HF) of the RTGLE distribution. Section \ref{sec4} deals with statistical properties, like, the shape of the PDF and HF, quantile function and its applications, moments and generating function, order and record statistics, R\'enyi entropy. In Section \ref{sec5}, the maximum likelihood estimators, least squares and weighted least squares estimators, Anderson-Darling estimators, Cramer-von Mises estimators of the parameters are formulated and their biases and mean square errors are calculated. Section \ref{sec6} explores the Monte Carlo simulation to exhibit the applicability of the proposed model. Finally, in Section \ref{sec7}, the simulation results are supplemented with numerical demonstrations using real data.

\section{Genesis of the Record-Based Transmuted Distributions}\label{sec2}
	The record-based transmuted map has a simple form, given in (\ref{eq3}), to generate new probability models and is proposed by Balakrishnan and He (2021). The construction of this map is as follows: Let $X_{1}, X_{2}, \ldots $ be a sequence of independent and identically distributed (IID) random variables with DF $G(\cdot)$. Let $X_{U(1)}$ and $X_{U(2)}$ be the first two  upper records from this sequence of IID random variables (for more details, see an excellent book by Arnold et al. (1998)). Define a random variable $Y$ as 
	\begin{align*}
		Y=\begin{cases}
			X_{U(1)},& \text { with probability } 1-p \\
			X_{U(2)},& \text { with probability } p,
		\end{cases}
	\end{align*}
	where $p \in [0, 1]$. Then, the DF of $Y$ can be obtained as 
	\begin{align*}\label{eq1.4}
		F_{Y}(x) &=(1-p) P\left(X_{U(1)} \leq x\right)+p P\left(X_{U(2)} \leq x\right) \\
		&=(1-p) G(x)+p \left[1-\bar{G}(x) \sum_{k=0}^{1} \frac{(-\log \bar{G}(x))^{k}}{k!}\right] \\
		&=(1-p) G(x)+p [1-\bar{G}(x)(1-\log \bar{G}(x))] \\
		&=(1-p) G(x)+p [G(x)+\bar{G}(x)\log\bar{G}(x)]\\
		&=G(x)+p\bar{G}(x)\log\bar{G}(x),\; x \in \mathbb{R}, \numberthis
	\end{align*}
	where $\bar{G}(x)=1-G(x)$ denotes the SF of the baseline distribution. The probability density function (PDF) and the hazard function (HF) are, respectively, given by 
	\begin{equation}\label{eq1.5}
		f_{Y}(x)=g(x)[1-p-p\log\bar{G}(x))], \; x \in \mathbb{R},
	\end{equation}
	and
	\begin{equation}\label{eq1.6}
		h_{Y}(x)=h_{X}(x) \frac{1-p-p \log\bar{G}(x)}{1+p \log \bar{G}(x)}, \; x \in \mathbb{R},
	\end{equation}
	where $g(x)$ is the PDF of the baseline distribution, and  $h_{X}(x)=\displaystyle \frac{g(x)}{\bar{G}(x)}$ is the HF of the baseline distribution.
	
	
	\section{Record-Based Transmuted Generalized Linear Exponential Distribution}\label{sec3}
	This section deals with the RTGLE distribution. Let us consider the  generalized linear exponential (GLE) distribution with parameters  $(\alpha,\beta,\gamma)$ as the baseline distribution. The PDF, DF and HF of the GLE$(\alpha, \beta,\gamma)$ distribution are, respectively, defined below
	\begin{equation*}
		g(x)=\gamma (\alpha+\beta x)\left(\alpha x+\frac{\beta x^{2}}{2} \right)^{\gamma-1} e^{-\left(\alpha x+\frac{\beta x^{2}}{2} \right)^{\gamma}},
		\end{equation*}

	\begin{equation*}
		G(x)=1-e^{-\left(\alpha x+\frac{\beta x^{2}}{2}\right)^{\gamma}},
	\end{equation*}
	and
	\begin{equation}\label{BHRF}
		k(x)=\gamma(\alpha+\beta x)\left(\alpha x+\frac{\beta x^{2}}{2}\right)^{\gamma-1},
	\end{equation}
	where $x>0$, $\alpha\ge 0, \beta \geq 0$ and $\gamma>0$, but $\alpha$ and $\beta$ cannot be zero together.  
	The DF, PDF and HF of RTGLE$(\alpha,\beta,\gamma,p)$ distribution constructed by substituting the above expression in  \eqref{eq1.4}-\eqref{eq1.6}, we get
	\begin{equation}\label{eq2.3}
		F(x) =1-\left[1+p\left(\alpha x+\frac{\beta x^{2}}{2}\right)^{\gamma}\right]e^{-\left(\alpha x+\frac{\beta x^{2}}{2}\right)^{\gamma}},
	\end{equation}
	\begin{equation}\label{PDF1}
		f(x) =\gamma(\alpha+\beta x)\left(\alpha x+\frac{\beta x^{2}}{2}\right)^{\gamma-1}
		\left[1-p+p\left(\alpha x+\frac{\beta x^{2}}{2}\right)^{\gamma}\right]
		e^{-\left(\alpha x+\frac{\beta x^{2}}{2}\right)^{\gamma}},
	\end{equation}
	and 
	\begin{equation}\label{HRF}
		h(x)=\frac{\gamma(\alpha+\beta x)\left(\alpha x+\frac{\beta x^{2}}{2}\right)^{\gamma-1}\left[1-p+p\left(\alpha x+\frac{\beta x^{2}}{2}\right)^{\gamma}\right]}{1+p\left(\alpha x+\frac{\beta x^{2}}{2}\right)^{\gamma}},
	\end{equation}
	where $x>0$, $\alpha\ge 0, \beta \geq 0, \gamma>0$, and $p \in [0,1]$, but $\alpha$ and $\beta$ cannot be zero together. If $p=0$, the DF $F(x)$ of RTGLE$(\alpha,\beta,\gamma,p)$ distribution becomes the DF $G(x)$ of GLE$(\alpha,\beta,\gamma)$ (baseline) distribution. It can be easily verify that $F(x) \leq G(x), \; \forall x>0.$ This implies that the tail probabilities under RTGLE distribution are larger than or equal to the tail probabilities under GLE distribution. Several well known probability models can be reduced from the RTGLE distribution, and some possible well-known sub-models are reported in Table \ref{table1}. 
	\begin{table}[htbp]
		\centering
		\caption{Some well known probability models reduced from RTGLE$(\alpha,\beta,\gamma,p)$ distribution}
		\scalebox{0.88}{
		\begin{tabular}{|l|l|l|l|}
			\hline
			S.No.	& Model & Survival Function $\displaystyle \bar{F}(x)=1-F(x)$  & Values of parameters  \\ 
			\hline
			 1. & Exponential  & $e^{-\alpha x}$ & $p=0, \beta=0, \gamma=1$ \\
			 2. &Rayleigh  & $e^{-\frac{\beta x^{2}}{2}}$ & $p=0, \alpha=0, \gamma=1$ \\
			 3. & Weibull & $e^{-\left(\alpha x\right)^{\gamma}}$ & $p=0, \beta=0$ \\
			 4. & Linear exponential  & $e^{-\left(\alpha x+\frac{\beta x^{2}}{2}\right)}$ & $p=0, \gamma=1$ \\
			 5. &Generalized linear exponential  & $e^{-\left(\alpha x+\frac{\beta x^{2}}{2}\right)^{\gamma}}$ & $p=0$ \\
			 6.&	RT-exponential  & $\left[1+p\alpha x \right]e^{-\alpha x}$ & $\beta=0,\gamma=1$ \\
			 7. &	RT-Rayleigh  & $\left[1+\frac{p \beta x^{2}}{2} \right]e^{-\frac{\beta x^{2}}{2}}$ & $\alpha=0, \gamma=1$ \\
				8. & RT-Weibull  & $\left[1+p\left(\alpha x\right)^{\gamma} \right]e^{-\left(\alpha x\right)^{\gamma}}$ & $\beta=0$ \\
			 9. &RT-Linear exponential  & $\left[1+p\left(\alpha x+\frac{\beta x^{2}}{2}\right) \right]e^{-\left(\alpha x+\frac{\beta x^{2}}{2}\right)}$ & $\gamma=1$ \\
			\hline
			
		\end{tabular}
	}
\label{table1}
	\end{table}
	\newpage
	\section{Statistical Properties}\label{sec4}
	\subsection{Shapes of the Probability Density and Hazard Functions}
	This subsection deals with the shapes of the PDF and HF of RTGLE$(\alpha,\beta,\gamma,p)$ distribution. For sufficiently small $x$, i.e., $x \rightarrow 0$, we have 
	$$f(x) \sim \alpha \gamma (1-p) \left(\alpha x+ \frac{\beta x^2}{2}\right)^{\gamma-1}, \; p \neq 1,$$
	and, for sufficiently large $x$, i.e., $x \rightarrow \infty$, we have 
	$$f(x) \sim \gamma p (\alpha+\beta x) \left(\alpha x+ \frac{\beta x^2}{2}\right)^{2\gamma-1} e^{-\left(\alpha x+ \frac{\beta x^2}{2}\right)^{\gamma}}, \; p \neq 0.$$
	It may notice that the lower tails of the PDF grow polynomially  while the upper tails decay exponentially. The PDF plots for different values of parameters $(\alpha,\beta,\gamma,p)$ are shown in Figure \ref{FigPDF}. Moreover, the following theorem provides the shape of the PDF of the RTGLE distribution.  
	\begin{theorem}\label{Thm1}
		The PDF of RTGLE$(\alpha,\beta,\gamma,p)$ distribution is unimodal if $\gamma \geq 1$, and is monotonically decreasing if $0<\gamma<0.5$ and $\beta=0$.
	\end{theorem}
	\begin{proof}
		Let $z=\left(\alpha x+ \frac{\beta x^2}{2}\right)^{\gamma}$. On solving the quadratic equation for $x$, we get  $x=\frac{1}{\beta} \Bigl\{(2 \beta z^{\frac{1}{\gamma}}+\alpha^2)^{\frac{1}{2}}-\alpha\Bigr\}$. Since $x>0$, it follows that $z>0$. Now, we can write the PDF of RTGLE$(\alpha,\beta,\gamma,p)$ distribution as a function of $z$, and so
		$$\eta(z)=f\left(\frac{1}{\beta} \left\{(2 \beta z^{1/\gamma}+\alpha^2)^{1/2}-\alpha\right\}\right)=\gamma z^{\frac{\gamma-1}{\gamma}}(2 \beta z^{1/\gamma}+\alpha^2)^{1/2}
		(1-p+pz) e^{-z}.$$ 
		The first and second orders partial derivatives of $\ln(\eta(z))$ are, respectively, given by
		\begin{equation}\label{Pderivative}
			\frac{\partial \ln(\eta(z)) }{\partial z}= \frac{\beta z^{\frac{1}{\gamma}-1}}{\gamma (2 \beta z^{1/\gamma}+\alpha^2)}+ \left(1-\frac{1}{\gamma}\right) \frac{1}{z} +\frac{p}{(1-p+pz)} -1,
		\end{equation}
	and 
	\begin{equation}\label{Pderivative2}
		\frac{\partial^2 \ln(\eta(z)) }{\partial z^2}=-\frac{\beta z^{\frac{1}{\gamma}-2} \left\{\gamma\left( 2 \beta z^{\frac{1}{\gamma}} +\alpha^2\right)-\alpha^2\right\}}{\gamma^2\left( 2 \beta z^{\frac{1}{\gamma}} +\alpha^2\right)^2}-\left(1-\frac{1}{\gamma}\right) \frac{1}{z^2}-\frac{p^2}{(1-p+pz)^2}.
	\end{equation}
	For $\gamma \geq 1$, it follows from the equations (\ref{Pderivative}) and (\ref{Pderivative2}) that $\frac{\partial \ln(\eta(z)) }{\partial z}$ is a decreasing function of $z$ with $\lim_{z \rightarrow 0} \frac{\partial \ln(\eta(z)) }{\partial z}=+\infty$ and $\lim_{z \rightarrow \infty} \frac{\partial \ln(\eta(z)) }{\partial z}=-1<0$. This implies that the partial derivative of $\ln(\eta(z))$ will be changing sign positive to negative. Thus, the PDF of RTGLE distribution is unimodal if $\gamma \geq 1$. Similarly, for $0<\gamma<0.5$ and $\beta=0$, it is easy to verify that $\frac{\partial \ln(\eta(z)) }{\partial z}$ is an increasing function of $z$ with $\lim_{z \rightarrow 0} \frac{\partial \ln(\eta(z)) }{\partial z}=-\infty$ and $\lim_{z \rightarrow \infty} \frac{\partial \ln(\eta(z)) }{\partial z}=-1<0$. Hence, the PDF is decreasing in nature if $0<\gamma<0.5$ and $\beta=0$.     
	\end{proof}
\begin{remark}
	Note that Theorem \ref{Thm1} is unable to provide the shape of the PDF of RTGLE distribution when $\beta \neq 0$ and  $\gamma \in (0, 1)$; or $\beta=0$ and $\gamma \in [0.5, 1)$. We have plot the PDF of RTGLE distribution in Figure \ref{FigPDF} for different choices of parametric values. It can be observed from Figure \ref{FigPDF} that the PDF of RTGLE distribution is monotonically decreasing if $\gamma \in (0, 1)$. Thus, we conjectured that the shape of the PDF of RTGLE distribution depends on the parameter $\gamma$ only.   
\end{remark}
	Now, we will discuss the behavior of the HF given in (\ref{HRF}).  For sufficiently small $x$, i.e., $x \rightarrow 0$, we get
	$h(x) \sim \alpha \gamma  \left(\alpha x+ \frac{\beta x^2}{2}\right)^{\gamma-1}, $
	and, for sufficiently large $x$, i.e., $x \rightarrow \infty$, we get
	$h(x) \sim \gamma (\alpha+\beta x) \left(\alpha x+ \frac{\beta x^2}{2}\right)^{\gamma-1}.$ This implies that the lower and upper tails of the HF of RTGLE$(\alpha,\beta,\gamma,p)$ distribution behave polynomially. It is evident from equations (\ref{BHRF}) and (\ref{HRF}) that $h(x) \leq k(x)$ for all $x>0$. This implies that the record-based transmutation process has a dampening effect on the HF of the baseline distribution (see Balakrishnan and He (2021)). The HF plots for different values of parameters $(\alpha,\beta,\gamma,p)$ are given in Figure \ref{hf}. These plots show that RTGLE distribution has IFR, DFR, and BTFR. Moreover, the following theorem proves that the RTGLE$(\alpha,\beta,\gamma,p)$ distribution has IFR if $\gamma \geq 1$, and has DFR if $\beta=0, \gamma \in (0,0.5]$. 
	\begin{figure}[t]
		\caption{Graph of PDF}
		\centering
		\includegraphics[width=1\linewidth, height=8cm]{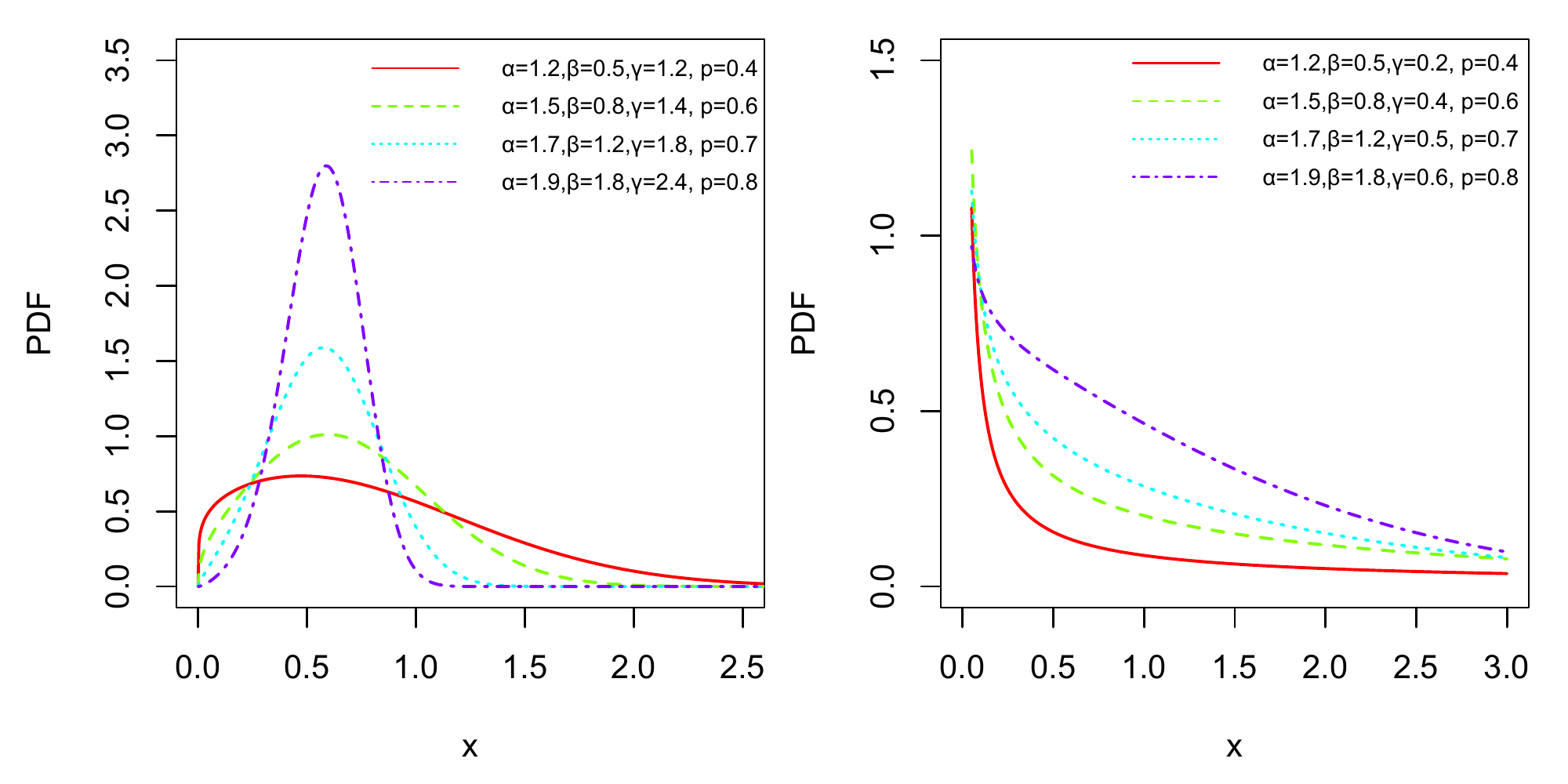} 
		\label{FigPDF}
	\end{figure}
	\begin{theorem}\label{thm2}
		Let $X$ be a random variable having RTGLE$(\alpha,\beta,\gamma,p)$ distribution, and the HF $h(x)$ of $X$ is given in equation (\ref{HRF}). Then, \\
		(i) the HF is an increasing function if $\gamma \geq 1$,\\
		(ii) the HF is a decreasing function if $\beta=0$ and $0<\gamma < 0.5$.
	\end{theorem}
	\begin{proof}
		(i)	Let $\mathcal{U}(z)=-\frac{\partial^2 \ln(\eta(z)) }{\partial z^2}$, where $\eta(z)$ is the same as in Theorem \ref{Thm1}. Now, using equation (\ref{Pderivative2}), we get
		\begin{equation}\label{uz1}
			\mathcal{U}(z)=\frac{\beta z^{\frac{1}{\gamma}-2} \left\{\gamma\left( 2 \beta z^{\frac{1}{\gamma}} +\alpha^2\right)-\alpha^2\right\}}{\gamma^2\left( 2 \beta z^{\frac{1}{\gamma}} +\alpha^2\right)^2}+\left(1-\frac{1}{\gamma}\right) \frac{1}{z^2}+\frac{p^2}{(1-p+pz)^2}.
		\end{equation}
		From equation (\ref{uz1}), $\mathcal{U}(z)>0$ if $\gamma \geq 1$. So, it follows from Glaser (1980) result that the HF is an increasing function if $\gamma \geq 1$. \\
		(ii) If $\beta=0$, then the equation (\ref{uz1}) reduces as
		\begin{equation}\label{uz2}
			\mathcal{U}(z;\beta=0)=\left(1-\frac{1}{\gamma}\right) \frac{1}{z^2}+\frac{p^2}{(1-p+pz)^2}.
		\end{equation} 	  
		It can be verify from equation (\ref{uz2}) that $\mathcal{U}(z, \beta=0)<0$ if $\gamma <0.5$. Thus, the result follows from the result of Glaser (1980).
	\end{proof}

	As a consequence of Theorem \ref{thm2} (i), the RTGLE distribution belongs in several well known ageing classes. Some of the implications among these ageing classes are given below: 
	\begin{equation*}
		\begin{array}{ccccccccc}
			\mathrm{IFR}& \Rightarrow &\mathrm{IFRA}& \Rightarrow&\mathrm{NBU}& \Rightarrow& \mathrm{NBUFR}& \Rightarrow& \mathrm{NBUFRA} \\
			 &&&&\Downarrow&&&& \\
			\Big\Downarrow &&&&\mathrm{NBUC}&&&& \\
			 &&&&\Downarrow&&&& \\
			\mathrm{DMRL} &&\Rightarrow&& \mathrm{NBUE} &\Rightarrow& \mathrm{HNBUE}& \Rightarrow & \mathcal{L}
		\end{array}	
	\end{equation*}
	Here, IFRA, NBU, NBUFR, NBUFRA, NBUC, DMRL, NBUE, HNBUE, and $\mathcal{L}$ denote increasing failure rate average, new better than used, new better than used in failure rate, new better than used in failure rate average, new better than used in convex ordering, decreasing mean residual life, new better than used in expectation, harmonically new better than used in expectation, and Laplace class, respectively. Similar implications exist for DFR class of distributions. For more details, the readers are advised to see Deshpande et al. (1986), Kochar and Wiens (1987), Lai and Xie (2006). Thus, the RTGLE distribution may be used for modelling the lifetime data having different ageing properties.  
	\begin{figure}[H]
		\caption{Graph of HF}
		\centering
		\label{hf}
		\includegraphics[width=1\linewidth, height=9cm]{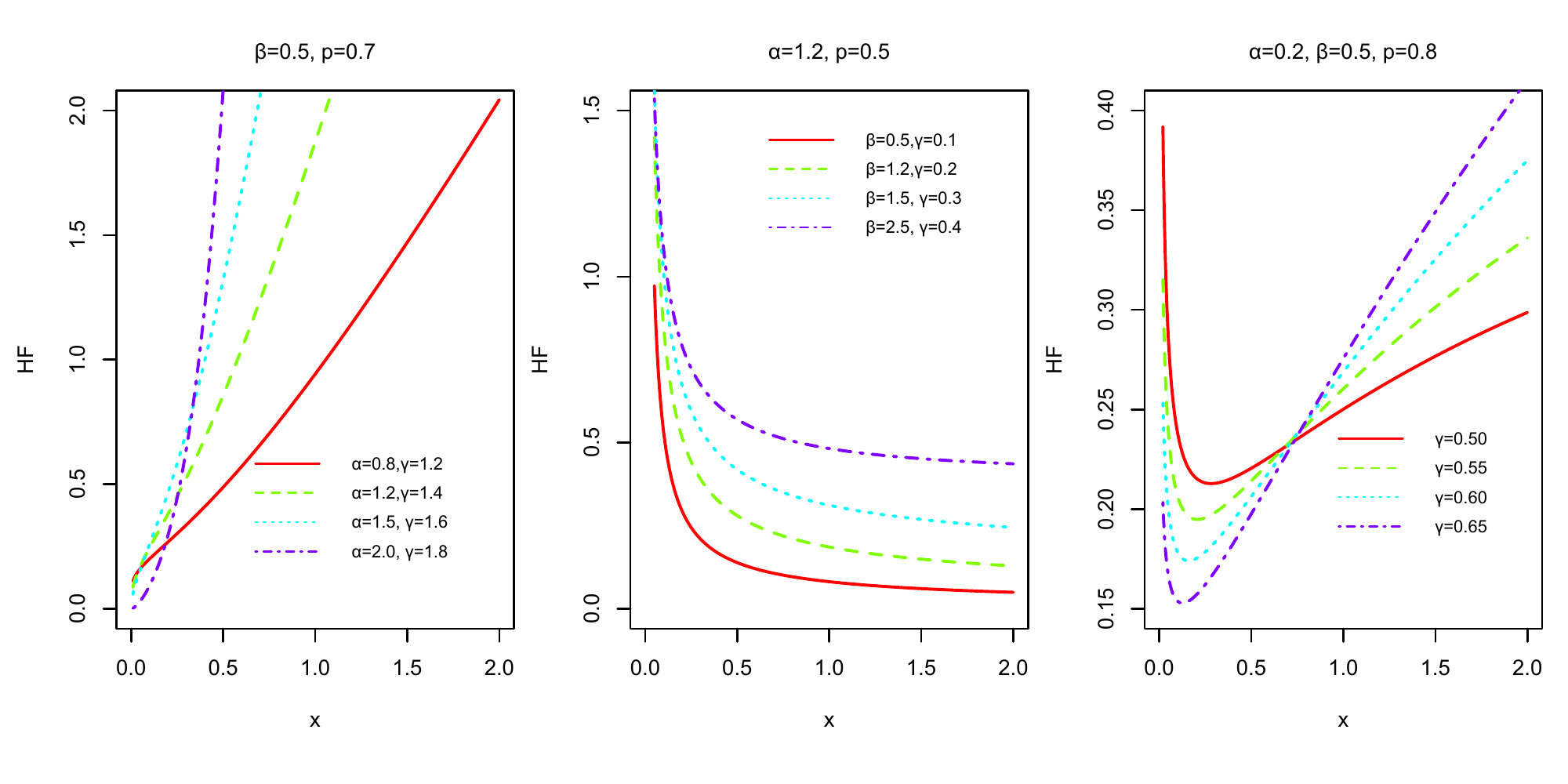} 
	\end{figure}      
	\subsection{Quantile Function and Its Applications}
	The good news is that the RTGLE distribution has a closed-form expression for quantile function (QF). The closed-form of the QF can be derived with the help of the Lambert $W$ function. The Lambert $W$ function is defined as the solution of the equation 
	\begin{equation}\label{W-function}
		W(\mathcal{V}) e^{W(\mathcal{V})}=\mathcal{V},
	\end{equation}
	where $\mathcal{V}$ is a complex number. The equation (\ref{W-function}) has an infinite number of solutions and so is a multi-valued function. The different possible solutions are called the branches. For real number $\mathcal{V} \geq -1/e$, the Lambert $W$ function has only two branches, which are denoted by $W_0$ (principal branch) and $W_{-1}$ (negative branch). For $\mathcal{V} \geq -1/e$, the principal branch $W_0$ takes values in $[-1, \infty)$, and for $\mathcal{V} \in [-1/e,0)$, the negative branch $W_{-1}$ takes values in $(-\infty,-1]$.  A large number of applications of the Lambert $W$ function are available in the literature (see Jodra and Arshad (2021)). For a detailed discussion on it, we refer to the article by Corless et al. (1996). Now, we state the QF of the RTGLE distribution. 
	\begin{theorem}
		Let $X$ be a random variable having RTGLE$(\alpha,\beta,\gamma,p)$ distribution. The quantile function $Q(u)$ of $X$ is given by 
		\begin{equation}\label{Quantile}
			Q(u)=\frac{-\alpha+\sqrt{\alpha^{2}+2 \beta C_{p,u,\gamma}}}{\beta}, \; 0<u<1,
		\end{equation}
		where
		\begin{equation}\label{Cfunction}
			C_{p,u,\gamma}=\left(-\frac{1}{p}\left[1+p W_{-1}\left(\frac{u-1}{p e^{\frac{1}{p}} }\right)\right]\right)^{\frac{1}{\gamma}}
		\end{equation}
		and $W_{-1}(\cdot)$ denotes the negative branch of the Lambert $W$ function. 
	\end{theorem} 
	\begin{proof}
		For $u \in (0, 1)$, we need to solve the equation $F(x)=u,$ where $F(\cdot)$ is the DF given in equation (\ref{eq2.3}). This is  equivalent to solving the following equation
		\begin{align*}
			\left[1+p\left(\alpha x+\frac{\beta x^{2}}{2}\right)^{\gamma}\right]e^{-\left(\alpha x+\frac{\beta x^{2}}{2}\right)^{\gamma}}&=1-u, \\
			\text{or} \hspace{2cm}\;\; 
			\left[-\frac{1}{p}-\left(\alpha x+\frac{\beta x^{2}}{2}\right)^{\gamma}\right] e^{-\left(\alpha x+\frac{\beta x^{2}}{2}\right)^{\gamma}-\frac{1}{p}}&=\frac{u-1}{pe^{\frac{1}{p}}}. 
		\end{align*}
		It can be verified that $(u-1)/pe^{\frac{1}{p}} \in [-1/e,0)$, and $-\frac{1}{p}-\left(\alpha x+\frac{\beta x^{2}}{2}\right)^{\gamma} \in (-\infty,-1]$. Now, using the negative branch of the Lambert $W$ function in the above equation, we get
		\begin{equation}\label{QFLWF}
			W_{-1}\left(\frac{u-1}{pe^{\frac{1}{p}}}\right)=-\frac{1}{p}-\left(\alpha x+\frac{\beta x^{2}}{2}\right)^{\gamma}.
		\end{equation}
		On solving the equation (\ref{QFLWF}), and using the function $C_{p,u,\gamma}$ given in equation (\ref{Cfunction}), we have
		\begin{equation}\label{Quareq}
			\frac{\beta}{2} x^2+\alpha x- C_{p,u,\gamma}=0.
		\end{equation}
		Finally, the QF, given in equation (\ref{Quantile}), is obtained by solving the quadratic equation (\ref{Quareq}) with respect to $x$.  	
	\end{proof}
	It is worth mentioning that the Lambert $W$ function is available in several computing platforms such as $R$-software, Matlab, Mathematica, and Maple, etc. Therefore, random sample generation from the RTGLE distribution is straightforward from equations (\ref{Quantile}) and (\ref{Cfunction}). The QF can be used as an alternative to the DF. It has several interesting properties that are not shared by the DF, e.g., the sum of two QFs is again a QF. Quantile-based data analysis is a burn-in topic in statistical literature. Specially, in reliability engineering and life testing experiments, a lot of measures have been developed using the QF. Some of the important measures based on QF are listed below:  
\begin{itemize}
	\item A measure of location is the median defined by $Q(0.5)$;
	\item Three quartiles $Q_1=Q(0.25), Q_2=Q(0.5)$ and $Q_3=Q(0.75)$ are useful to describe the summary of distribution and making the Box Plot to detect the outliers;   
	\item Dispersion is measured by the interquartile range $IQR=Q_3 -Q_1$;
	\item Skewness is measured by Galton's coefficient $GC=(Q_1+Q_3-2Q_2)/(Q_3-Q_1)$, which lies between $-1$ and $+1$. For symmetric distribution, $GC=0$;
	\item Moors coefficient of kurtosis is defined by 
	$$MC=\frac{Q(0.875)-Q(0.625)+Q(0.375)-Q(0.125)}{Q_3 -Q_1}.$$  
\end{itemize}
Given the QF (in equation (\ref{Quantile})), the calculations of all the above measures are very simple. We have computed these measures by using the $R$-software (version 4.0.0; R-core team 2021), and reported in Table \ref{tableQF}. Apart from these measures, there are several other quantile-based measures studied in the literature, e.g., Gini's mean difference defined as $\Delta=2 \int_{0}^{1} (2u-1) Q(u) du$; the $r$th $L$-moment is defined as 
\begin{equation}
	L_r=\sum_{k=0}^{r-1} (-1)^{r-1-k} {r-1 \choose k} {r-1+k \choose k} \int_{0}^{1} u^k Q(u) \ du. \nonumber 
\end{equation}
The applications of $L$-moments in reliability theory suggest that we can use these as a competing alternative to conventional moments. The $L$-moments are able to provide a summary statistics of distribution, to identify the distribution and to fit models to data. For further reading, the reader can see an excellent book by Nair et al. (2013).	

\begin{table}[H]
	\centering
	\caption{Some measures based on quantile function}
	\scalebox{0.85}{
		\begin{tabular}{|ccc|c|cccc|}
			\hline
			$\beta$    & $\gamma$    & $p$     & $\alpha$   & Median & IQR   & Skewness & Kurtosis \\
			\hline
			&       &       & 0.5   & 1.1199 & 1.0325 & 0.0728 & 1.0866 \\
			&       &       & 1.0     & 0.7375 & 0.7996 & 0.1211 & 1.0959 \\
			&       &       & 1.5   & 0.5347 & 0.6220 & 0.1494 & 1.1149 \\
			0.5   & 1.2   & 0.2   & 2.0     & 0.4152 & 0.4999 & 0.1646 & 1.1294 \\
			&       &       & 2.5   & 0.3380 & 0.4145 & 0.1731 & 1.1391\\
			&       &       & 3.0     & 0.2844 & 0.3527 & 0.1783 & 1.1455 \\
			&       &       & 3.5   & 0.2453 & 0.3063 & 0.1817 & 1.1499 \\
			\hline
			$\alpha$    & $\gamma$    & $p$     & $\beta$   & Median & IQR   & Skewness & Kurtosis \\
			\hline
			&       &       & 0.5   & 1.1199 & 1.0325 & 0.0728 & 1.0866 \\
			&       &       & 1.0     & 0.9131 & 0.7760 & 0.0571 & 1.0908\\
			&       &       & 1.5   & 0.7962 & 0.6479 & 0.0509 & 1.0937 \\
			0.5   & 1.2     & 0.2   & 2.0     & 0.7175 & 0.5677 & 0.0475 & 1.0956\\
			&       &       & 2.5   & 0.6595 & 0.5114 & 0.0454 & 1.0969\\
			&       &       & 3.0     & 0.6144 & 0.4691 & 0.0439 & 1.0978  \\
			&       &       & 3.5   & 0.5779 & 0.4358 & 0.0429 & 1.0985\\
			\hline
			$\alpha$    & $\beta$    & $p$     & $\gamma$   & Median & IQR   & Skewness & Kurtosis \\ 
			\hline 
			&       &       & 0.5   & 0.9726 & 2.2683 & 0.3437 & 1.2406 \\
			&       &       & 1.0     & 1.0978 & 1.2174 & 0.1127 & 1.0900 \\
			&       &       & 1.5   & 1.1424 & 0.8419 & 0.0330 & 1.0897 \\
			0.5   & 0.5   & 0.2   & 2.0     & 1.1652 & 0.6446 & -0.0066 & 1.0992 \\
			&       &       & 2.5   & 1.1791 & 0.5225 & -0.0302 & 1.1080\\
			&       &       & 3.0     & 1.1885 & 0.4394 & -0.0459 & 1.1151 \\
			&       &       & 3.5   & 1.1952 & 0.3791 & -0.0571 & 1.1208 \\
			\hline
			$\alpha$    & $\beta$    & $\gamma$     & $p$   & Median & IQR   & Skewness & Kurtosis \\
			\hline
			&       &       & 0.1   & 1.0511 & 0.9928 & 0.0813 & 1.0900 \\
			&       &       & 0.2   & 1.1199 & 1.0325 & 0.0728 & 1.0866  \\
			&       &       & 0.3   & 1.1922 & 1.0621 & 0.0617 & 1.0847 \\
			0.5   & 0.5   & 1.2   & 0.5   & 1.3413 & 1.0851 & 0.0390 & 1.0924 \\
			&       &       & 0.7   & 1.4856 & 1.0612 & 0.0281 & 1.1104\\
			&       &       & 0.9   & 1.6163 & 1.0074 & 0.0332 & 1.1190\\
			&       &       & 1     & 1.6755 & 0.9761 & 0.0399 & 1.1173\\
			\hline
		\end{tabular}
	}
	\label{tableQF}
\end{table}
	\subsection{Moments and Its Generating Function}
	This section provides the formulas for $r^{th}$ moment, and the moments generating function (MGF) of RTGLE distribution. Let $X$ be a random variable having the RTGLE$(\alpha,\beta,\gamma,p)$ distribution, then $r^{th}$ moment about origin is defined as $\mu_{r}^{'}=E(X^{r})=\int_{0}^{\infty}x^r f(x)dx$, where $f(x)$ denotes the PDF given in equation (\ref{PDF1}). Now, set $z=\left(\alpha x+\frac{\beta x^{2}}{2}\right)^{\gamma}$, which implies that $x=\frac{1}{\beta}\left(\left(2 \beta z^{\frac{1}{\gamma}}+\alpha^{2}\right)^{\frac{1}{2}}-\alpha\right)$, and $dz=\gamma(\alpha+\beta x)\left(\alpha x+\frac{\beta x^{2}}{2}\right)^{\gamma-1}dx$. Using these values, we get
	\begin{align}
		\mu_{r}^{\prime}&=\frac{1}{\beta^r}\int_{0}^{\infty}\left(\left(2 \beta z^{\frac{1}{\gamma}}+\alpha^{2}\right)^{\frac{1}{2}}-\alpha\right) ^r(1-p+pz)e^{-z}dz\nonumber\\
		&=\frac{1}{\beta^r}\int_{0}^{\infty}\sum_{i=0}^{r} (-1)^{i}\binom{r}{i}
		\alpha^{i}\left(2 \beta z^{\frac{1}{\gamma}}+\alpha^{2}\right)^{\frac{r-i}{2}}(1-p+pz)e^{-z}dz\quad(\text{using binomial expansion})\nonumber\\
		&=\frac{1}{\beta^r}\int_{0}^{\infty}\sum_{i=0}^{r} (-1)^{i}\binom{r}{i}
		\alpha^{i}\left(2 \beta z^{\frac{1}{\gamma}} \right)^{\frac{r-i}{2}} \left(1+\frac{\alpha^{2}}{2 \beta z^{\frac{1}{\gamma}}}\right)^{\frac{r-i}{2}}(1-p+pz)e^{-z}dz\nonumber
	\end{align}
	Letting $\tau=\frac{2\beta}{\alpha^2}$, one can verify that $\tau z^{1 / \gamma}<1$ if $0<z<\tau^{-\gamma}$, and $\left(\tau z^{1 / \gamma}\right)^{-1}<1$ if $\tau^{-\gamma}<z<\infty$. So, we have
	\begin{align*}
		\mu_{r}^{\prime}&= \sum_{i=0}^{r}\binom{r}{i} \frac{(-\alpha)^{i}}{\beta^{r}} \int_{0}^{\tau^{-\gamma}}\left(1+\tau z^{1 / \gamma}\right)^{\frac{r-i}{2}}\alpha^{r-i} (1-p+pz)e^{-z}dz \\
		&\quad+\sum_{i=0}^{r}\frac{(-\alpha)^{i}}{\beta^{r}} \int_{\tau^{-\gamma}}^{\infty}\left(1+\frac{1}{\tau z^{1 / \gamma}}\right)^{\frac{r-i}{2}}\left(2 \beta z^{1 / \gamma}\right)^{\frac{r-i}{2}} (1-p+pz)e^{-z}dz\\
		&=\sum_{i=0}^{r}\sum_{j=0}^{\infty}(-1)^{i}\binom{r}{i}\binom{\frac{r-i}{2}}{j} \alpha^{r-2j}2^j\beta^{j-r}  \int_{0}^{\tau^{-\gamma}}(1-p+pz) z^{\frac{j}{\gamma}}e^{-z}dz\\
		&\quad+\sum_{i=0}^{r}\sum_{j=0}^{\infty}(-1)^{i}\binom{r}{i}\binom{\frac{r-i}{2}}{j} \alpha^{r-2j}2^j\beta^{j-r}\int_{\tau^{-\gamma}}^{\infty}(1-p+pz) z^{\frac{j}{\gamma}}e^{-z}dz\\
		&=\sum_{i=0}^{r}\sum_{j=0}^{\infty}(-1)^{i}\binom{r}{i}\binom{\frac{r-i}{2}}{j} \alpha^{r-2j}2^j\beta^{j-r}  \int_{0}^{\infty}(1-p+pz) z^{\frac{j}{\gamma}}e^{-z}dz\\
		&=\sum_{i=0}^{r}\sum_{j=0}^{\infty}(-1)^{i}\binom{r}{i}\binom{\frac{r-i}{2}}{j} \alpha^{r-2j}2^j\beta^{j-r}\left\lbrace(1-p)\Gamma\left(\frac{j}{\gamma}+1\right)+p\Gamma\left(\frac{j}{\gamma}+2\right) \right\rbrace,\numberthis\label{mean}   
	\end{align*}
	where $\Gamma(a)=\int_{0}^{\infty}t^{a-1}e^{-t}dt$ is gamma function. The numerical summary based on moments is reported in Table \ref{tableMVSK}. It can be observed from Table \ref{tableMVSK} that RTGLE distribution is useful to model the positive as well as negative skewed data. Some $3$D-plots of skewness and kurtosis are also reported in Figure \ref{FigSK}. These graphs show the effects of parameters on the shape of the distribution. Now, we will provide a recurrence relation among the moments, which is useful for computation purpose. 
	\begin{theorem}
		Let $X \sim$ RTGLE $(\alpha,\beta,\gamma,p)$. Then, we have
		\begin{equation}\label{reccurence1}
			\sum_{i=0}^{r}\binom{r}{i}\alpha^i\left(\frac{\beta}{2}\right)^{r-i} \mu_{2 r-i}^{\prime}=\left(1+\frac{pr}{\gamma}\right)\Gamma\left(\frac{r}{\gamma}+1\right),
		\end{equation}
		where $\Gamma(a)=\int_{0}^{\infty}t^{a-1}e^{-t}dt$ denotes the usual gamma function. 
	\end{theorem}
	\begin{proof}
	Consider 
		\begin{align*}
			E&\left[\left(\alpha X+\frac{\beta X^{2}}{2}\right)^{r}\right]\\ 
			&=\int_{0}^{\infty} \gamma(\alpha+\beta x)\left(\alpha x+\frac{\beta x^{2}}{2}\right)^{\gamma+r-1}
			\left[1-p+p\left(\alpha x+\frac{\beta x^{2}}{2}\right)^{\gamma}\right]
			e^{-\left(\alpha x+\frac{\beta x^{2}}{2}\right)^{\gamma}}
			\; dx.
		\end{align*}
	Using the transformation $z=\left(\alpha x+\frac{\beta x^{2}}{2}\right)^{\gamma}$, which implies that $dz=\gamma(\alpha+\beta x)\left(\alpha x+\frac{\beta x^{2}}{2}\right)^{\gamma-1}dx$. Therefore
		\begin{align*}
			E\left[\left(\alpha X+\frac{\beta X^{2}}{2}\right)^{r}\right] &=\int_{0}^{\infty} z^{\frac{r}{\gamma}} \left(1-p+p z\right)e^{-z}dz \\
			&=(1-p)\Gamma\left(\frac{r}{\gamma}+1\right)+p\Gamma\left(\frac{r}{\gamma}+2\right)\\
			&=\left(\frac{pr}{\gamma}+1\right)\Gamma\left(\frac{r}{\gamma}+1\right).\numberthis \label{3.11}
		\end{align*}
		Now expanding the left hand side of \eqref{3.11} binomially, we get
		\begin{align*}
			E\left[\sum_{i=0}^{r}\binom{r}{i}\alpha^i\left(\frac{\beta}{2}\right)^{r-i} X^{2r-i}\right]&=\left(\frac{pr}{\gamma}+1\right)\Gamma\left(\frac{r}{\gamma}+1\right) \\
			\sum_{i=0}^{r}\binom{r}{i}\alpha^i\left(\frac{\beta}{2}\right)^{r-i} \mu_{2r-i}^{\prime}&=\left(\frac{pr}{\gamma}+1\right)\Gamma\left(\frac{r}{\gamma}+1\right),
		\end{align*}
		which proves the result.
	\end{proof}

\begin{remark}
		For variance, we first obtain $\mu_{2}^{\prime}$ as follows. On substituting $r=1$ in \eqref{reccurence1}, we have
		\begin{align*}
			\frac{\beta}{2} \mu_{2}^{\prime}+\alpha \mu_{1}^{\prime}&=\left(\frac{p}{\gamma}+1\right)\Gamma\left(\frac{1}{\gamma}+1\right) \\ \\
			\mu_{2}^{\prime}&=\frac{2}{\beta}\left[\left(\frac{p}{\gamma}+1\right)\Gamma\left(\frac{1}{\gamma}+1\right)-\alpha \mu_{1}^{\prime}\right]
		\end{align*}
		Hence, the variance of $X$ is given by
		\begin{align*}
			\operatorname{Var}(X) &=\mu_{2}^{\prime}-\left(\mu_{1}^{\prime}\right)^2 \\
			&=\frac{2}{\beta}\left[\left(\frac{p}{\gamma}+1\right)\Gamma\left(\frac{1}{\gamma}+1\right)-\alpha \mu_{1}^{\prime}\right]-\left(\mu_{1}^{\prime}\right)^2,
		\end{align*}
		where $\mu_{1}^{\prime}=E(X)$, can be obtain by putting $r=1$ in equation \eqref{mean}.	
	\end{remark}
\begin{table}[H]
	\caption{Numerical values of mean, variance, skewness and kurtosis}
	\centering
		\scalebox{0.8}{
	\begin{tabular}{|ccc|c|ccccccc|}
		\hline
		$\beta$ & $\gamma$ & $p$ & $\alpha$ & $E(X)$ & $E(X^2)$ &$E(X^3)$ & $E(X^4)$ & $V(X)$ & $\gamma_1$ & $\beta_2$ \\ 
		\hline
		&  &  & 0.50 & 1.2058 & 1.9782 & 3.8832 & 8.6523 & 0.5243 & 0.6155 & 3.0496 \\ 
		&  &  & 1.00 & 0.8389 & 1.0343 & 1.5885 & 2.8406 & 0.3306 & 0.8738 & 3.5859 \\ 
		&  &  & 1.50 & 0.6300 & 0.6095 & 0.7541 & 1.1063 & 0.2126 & 1.0430 & 4.0914  \\ 
		0.50  &1.20  &0.20  & 2.00 & 0.4996 & 0.3932 & 0.4025 & 0.4950 & 0.1436 & 1.1518 & 4.4849 \\ 
		& &  & 2.50 & 0.4118 & 0.2714 & 0.2354 & 0.2474 & 0.1018 & 1.2233 & 4.7747 \\ 
		&  &  & 3.00 & 0.3494 & 0.1973 & 0.1478 & 0.1350 & 0.0753 & 1.2717 & 4.9856 \\ 
		&  & & 3.50 & 0.3029 & 0.1493 & 0.0982 & 0.0790 & 0.0576 & 1.3053 & 5.1403 \\ 
		\hline
		$\alpha$ & $\gamma$ & $p$ & $\beta$ & $E(X)$ & $E(X^2)$ &$E(X^3)$ & $E(X^4)$ & $V(X)$ & $\gamma_1$ & $\beta_2$ \\ 
		\hline
		& &  & 0.50 & 1.2058 & 1.9782 & 3.8832 & 8.6523 & 0.5243 & 0.6155 & 3.0496  \\ 
		&  &  & 1.00 & 0.9659 & 1.2290 & 1.8431 & 3.1093 & 0.2960 & 0.5232 & 2.9282 \\ 
		&  & & 1.50 & 0.8360 & 0.9059 & 1.1488 & 1.6320 & 0.2070 & 0.4817 & 2.8871  \\ 
		0.50 & 1.20 & 0.20 & 2.00 & 0.7503 & 0.7223 & 0.8102 & 1.0153 & 0.1594 & 0.4575 & 2.8672 \\ 
		&  &  & 2.50 & 0.6879 & 0.6028 & 0.6136 & 0.6966 & 0.1296 & 0.4414 & 2.8558 \\ 
		&  &  & 3.00 & 0.6396 & 0.5184 & 0.4869 & 0.5094 & 0.1093 & 0.4299 & 2.8486 \\ 
		&  &  & 3.50 & 0.6008 & 0.4554 & 0.3994 & 0.3897 & 0.0945 & 0.4212 & 2.8437 \\
		\hline
		$\alpha$ & $\beta$ & $p$ & $\gamma$ & $E(X)$ & $E(X^2)$ &$E(X^3)$ & $E(X^4)$ & $V(X)$ & $\gamma_1$ & $\beta_2$ \\ 
		\hline 
		&  &  & 0.50 & 1.7536 & 7.6929 & 51.3577 & 454.9978 & 4.6179 & 2.1839 & 9.7691 \\ 
		& &  & 1.00 & 1.2458 & 2.3083 & 5.3222 & 14.2781 & 0.7562 & 0.8550 & 3.5418 \\ 
		&  &  & 1.50 & 1.1789 & 1.7346 & 2.9302 & 5.4712 & 0.3447 & 0.3574 & 2.7217 \\ 
		0.50 & 0.50 & 0.20 & 2.00 & 1.1664 & 1.5666 & 2.3147 & 3.6749 & 0.2061 & 0.0711 & 2.6049 \\ 
		& &  & 2.50 & 1.1665 & 1.5000 & 2.0686 & 3.0119 & 0.1394 & -0.1205 & 2.6759 \\ 
		&  & & 3.00 & 1.1700 & 1.4701 & 1.9486 & 2.6951 & 0.1013 & -0.2592 & 2.8049 \\ 
		&  && 3.50 & 1.1743 & 1.4561 & 1.8833 & 2.5208 & 0.0772 & -0.3649 & 2.9486 \\ 
		\hline
		$\alpha$ & $\beta$ & $\gamma$ & $p$ & $E(X)$ & $E(X^2)$ &$E(X^3)$ & $E(X^4)$ & $V(X)$ & $\gamma_1$ & $\beta_2$ \\ 
		\hline 
		&  &  & 0.00 & 1.0771 & 1.6084 & 2.9069 & 6.0121 & 0.4483 & 0.6954 & 3.2167 \\ 
		&  &  & 0.20 & 1.2058 & 1.9782 & 3.8832 & 8.6523 & 0.5243 & 0.6155 & 3.0496 \\ 
		&  &  & 0.30 & 1.2701 & 2.1630 & 4.3713 & 9.9723 & 0.5498 & 0.5575 & 2.9545 \\ 
		0.50 & 0.50 & 1.20 & 0.50 & 1.3988 & 2.5328 & 5.3476 & 12.6125 & 0.5761 & 0.4412 & 2.8337 \\ 
		&  &  & 0.70 & 1.5275 & 2.9025 & 6.3239 & 15.2527 & 0.5693 & 0.3517 & 2.8243 \\ 
		&  &  & 0.90 & 1.6562 & 3.2723 & 7.3003 & 17.8929 & 0.5294 & 0.3305 & 2.9035\\ 
		&  & & 1.00 & 1.7205 & 3.4571 & 7.7884 & 19.2130 & 0.4970 & 0.3714 & 2.9450 \\ 
		\hline
	\end{tabular}
}\label{tableMVSK}
\end{table}
	\noindent We know that the MGF of a random variable can be written as $M_{X}(t)=\sum_{r=0}^{\infty}E(X^r)\frac{t^r}{r!}$. Using this formula, the following theorem provides the MGF of RTGLE distribution. 
	\begin{theorem}
		Let $X \sim$ RTGLE $(\alpha,\beta,\gamma,p)$. Then, the MGF of $X$ is given by 	
		\begin{align*}
			M_{X}(t)
			&=\sum_{r=0}^{\infty}\sum_{i=0}^{r}\sum_{j=0}^{\infty}(-1)^{i}\frac{t^r}{r!}\binom{r}{i}\binom{\frac{r-i}{2}}{j} \alpha^{r-2j}2^j\beta^{j-r}\left\lbrace(1-p)\Gamma\left(\frac{j}{\gamma}+1\right)+p\Gamma\left(\frac{j}{\gamma}+2\right) \right\rbrace. 
		\end{align*}
	\end{theorem}
	\begin{figure}[H]
		\caption{Graph of Skewness and Kurtosis}
		\centering
		
		\includegraphics[width=1\linewidth, height=7.5cm]{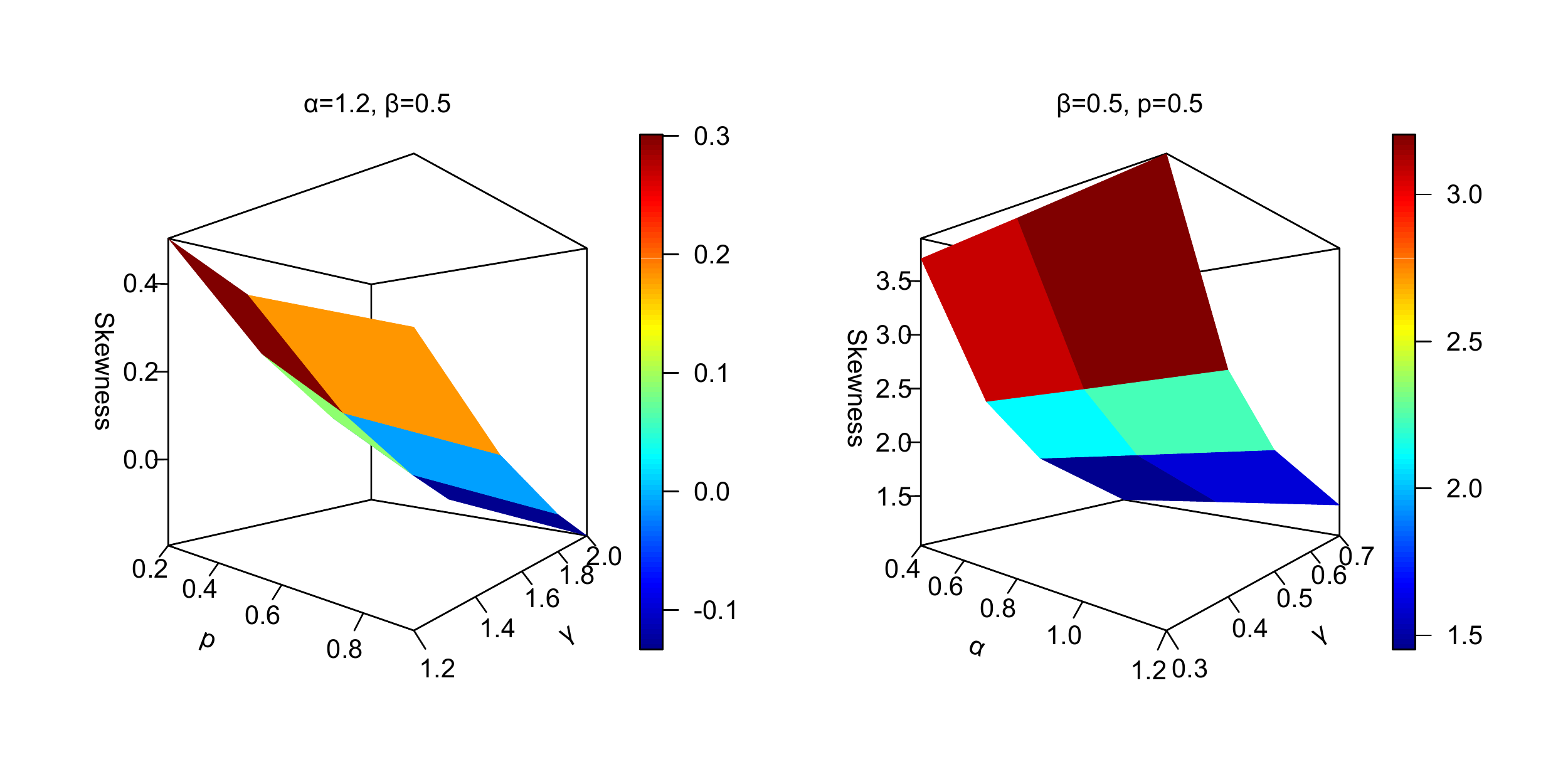} 
		\includegraphics[width=1\linewidth, height=7.5cm]{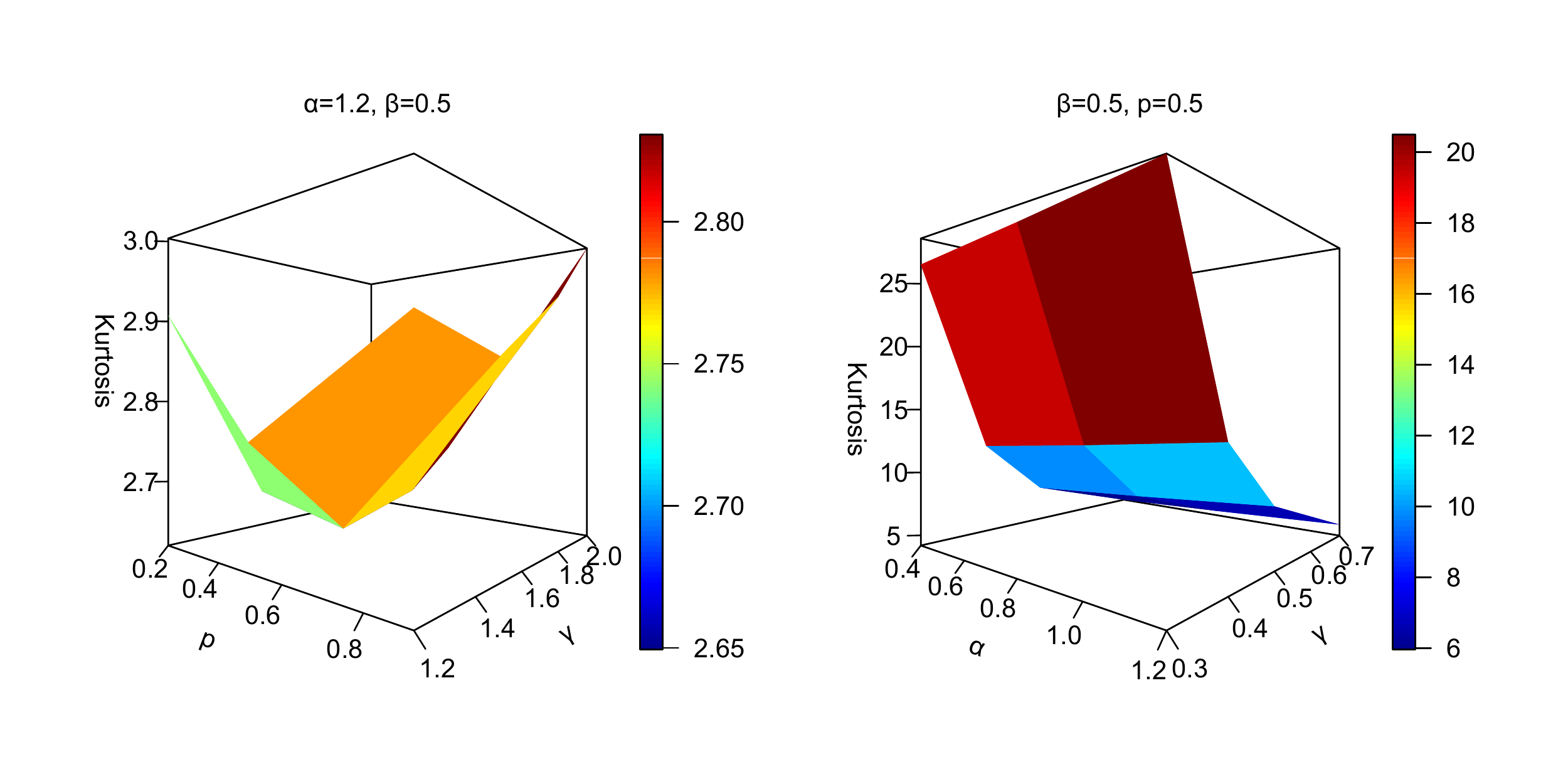} 
		\label{FigSK}
	\end{figure}  
	\subsection{Order and Record Statistics}
	This section provides the PDFs of ordered random variables when the random sample is drawn from RTGLE distribution. Order and record statistics are most commonly used models of ordered random variables and have applications in various fields of research. For more details, the readers are advised to see, Chandler (1952), Arnold et al. (1992, 1998), Ahsanullah (1995), David and Nagaraja. (2003), Khan and Arshad (2016), Arshad and Baklizi (2019), Arshad and Jamal (2019a, 2019b), Arshad et al. (2021), Azhad et al. (2021). Let $X_{1}, X_{2}, \ldots, X_{n}$ be a random sample taken from RTGLE distribution and let $X_{(1)} \leq X_{(2)} \leq \cdots \leq X_{(n)}$ denotes the corresponding order statistics. The PDF of $r$th order statistic $X_{(r)}$ ($r=1,2, \ldots,n$) is given by 
	\begin{align*}
		f_{X_{(r)}}(x) &=\frac{1}{B(r, n-r+1)} (F(x))^{r-1}(1-F(x))^{n-r} f(x) \\
		&=\frac{f(x)}{B(r, n-r+1)} \sum_{i=0}^{r-1}\binom{r-1}{i}(-1)^{i}(1-F(x))^{n+i-r}\\
		&=\frac{f(x)}{B(r, n-r+1)}  \sum_{i=0}^{r-1}\binom{r-1}{i}(-1)^{i}\left(1+p\left(\alpha x+\frac{\beta x^{2}}{2}\right)^{\gamma}\right)^{n+i-r}e^{-(n+i-r)\left(\alpha x+\frac{\beta x^{2}}{2}\right)^{\gamma}},
	\end{align*}
	where $f(\cdot)$ denotes the PDF of RTGLE distribution given in equation (\ref{PDF1}), and $B(a,b)$ denotes the usual beta function. Moreover, the PDFs of the smallest and the largest order statistics are, respectively, given below.
	\begin{equation*}
		f_{X_{(1)}}(x)=	n f(x) \left(1+p\left(\alpha x+\frac{\beta x^{2}}{2}\right)^{\gamma}\right)^{n-1}e^{-(n-1)\left(\alpha x+\frac{\beta x^{2}}{2}\right)^{\gamma}},
	\end{equation*}
	and
	\begin{equation*}
		f_{X_{(n)}}(x)=	n f(x)  \sum_{i=0}^{n-1}\binom{n-1}{i}(-1)^{i}\left(1+p\left(\alpha x+\frac{\beta x^{2}}{2}\right)^{\gamma}\right)^{i}e^{-i \left(\alpha x+\frac{\beta x^{2}}{2}\right)^{\gamma}}.
	\end{equation*}
	Let $X_1,X_2,\ldots$ be a sequence of random variables from the RTGLE distribution, and let $R_1, R_2,\ldots, R_n$ be the first $n$ upper record statistics observed from the sequence $X_1,X_2,\ldots$. Then, the PDF of $n$th record statistic $R_n$ is given by 
	\begin{equation*}
		f_{R_n}(r_n)=\frac{\left[-\log\left\{1+p\left(\alpha r_n+\frac{\beta r_{n}^2}{2}\right)^{\gamma}\right\} + \left(\alpha r_n+\frac{\beta r_{n}^2}{2}\right)^{\gamma} \right]^n}{n !} f(r_n), \;\; r_n>0,
	\end{equation*}
	and, the joint PDF of the first $n$ upper records $\boldsymbol{R}=(R_1,R_2,\ldots,R_n)$ is given by 
	\begin{equation*}
		f_{\boldsymbol{R}}(\boldsymbol{r})=	\left(\prod_{j=1}^{n-1} \frac{\gamma(\alpha+\beta r_j)\left(\alpha r_j +\frac{\beta r_{j}^{2}}{2}\right)^{\gamma-1}\left[1-p+p\left(\alpha r_j+\frac{\beta r_{j}^{2}}{2}\right)^{\gamma}\right]}{1+p\left(\alpha r_j+\frac{\beta r_{j}^{2}}{2}\right)^{\gamma}} \right) f(r_n), 
	\end{equation*}
	where $\boldsymbol{r}=(r_1,r_2,\ldots,r_n)$  denotes the observed value of $\boldsymbol{R}=(R_1,R_2,\ldots,R_n)$ with $r_1<r_2<\cdots < r_n$, and $f(\cdot)$ denotes the PDF given in equation (\ref{PDF1}).
	\subsection{R\'enyi Entropy}\label{sec9}
	The entropy of a random variable $X$ is a measure of the variation of uncertainty. There are many entopy types in literature. One of these entropy types is R\'enyi entropy. The R\'enyi entropy of the random variable $X$ is defined as 
	\begin{equation}\label{r}
		I_{\rho}(X)=\frac{1}{(1-\rho)} \log\left(  \int_{0}^{\infty} f(x)^{\rho}dx\right), 
	\end{equation}
	where $ \rho>0,\rho \neq 1$. Let $X$ be random variable having the RTGLE distribution with PDF given in (\ref{PDF1}). Now, set $z=\left(\alpha x+\frac{\beta x^{2}}{2}\right)^{\gamma}$, which implies that $x=\frac{1}{\beta}\left(\left(2 \beta z^{\frac{1}{\gamma}}+\alpha^{2}\right)^{\frac{1}{2}}-\alpha\right)$, and $dz=\gamma(\alpha+\beta x)\left(\alpha x+\frac{\beta x^{2}}{2}\right)^{\gamma-1}dx$. Using these values, we get 
	\begin{align*}
		&\int_{0}^{\infty} f(x)^{\rho}dx\\
		&=\int_{0}^{\infty}\left( \gamma(\alpha+\beta x)\left(\alpha x+\frac{\beta x^{2}}{2}\right)^{\gamma-1}
		\left[1-p+p\left(\alpha x+\frac{\beta x^{2}}{2}\right)^{\gamma}\right]
		e^{-\left(\alpha x+\frac{\beta x^{2}}{2}\right)^{\gamma}}\right)^{\rho}dx\\
		&=\gamma^{\rho-1}\int_{0}^{\infty}\left(\alpha^2+2\beta z^{\frac{1}{\gamma}} \right)^{\frac{\rho-1}{2}}(1-p+pz)^\rho z^{\frac{(\rho-1)(\gamma-1)}{\gamma}}e^{-\rho z}dz\\
		&=\gamma^{\rho-1}\sum_{i=0}^{\infty}\binom{\rho}{i}p^i(1-p)^{\rho-i}\int_{0}^{\infty}\left(\alpha^2+2\beta z^{\frac{1}{\gamma}} \right)^{\frac{\rho-1}{2}} z^{\frac{(\rho-1)(\gamma-1)}{\gamma}+i}e^{-\rho z}dz\\
		&=\gamma^{\rho-1}\sum_{i=0}^{\infty}\binom{\rho}{i}p^i(1-p)^{\rho-i}\Biggl\{\int_{0}^{\tau^{-\gamma}}\alpha^{\rho-1}\left(1+\tau z^{1/\gamma} \right)^{\frac{\rho-1}{2}} z^{\frac{(\rho-1)(\gamma-1)}{\gamma}+i}e^{-\rho z}dz\\&\quad+\int_{\tau^{-\gamma}}^{\infty}\left(1+\frac{1}{\tau z^{1/\gamma}} \right)^{\frac{\rho-1}{2}}(\beta z^{1/\gamma})^{\frac{\rho-1}{2}} z^{\frac{(\rho-1)(\gamma-1)}{\gamma}+i}e^{-\rho z}dz\Biggr\}\quad \text{ let }\left( \tau=\frac{2\beta}{\alpha^2}\right) \\
		&=\gamma^{\rho-1}\sum_{i=0}^{\infty}\binom{\rho}{i}p^i(1-p)^{\rho-i}\Biggl\{\sum_{j=0}^{\infty}\binom{\frac{\rho-1}{2}}{j}(2\beta)^j\alpha^{\rho-1-2j}\int_{0}^{\tau^{-\gamma}} z^{\frac{(\rho-1)(\gamma-1)+j}{\gamma}+i}e^{-\rho z}dz\\&\quad+\sum_{j=0}^{\infty}\binom{\frac{\rho-1}{2}}{j}(2\beta)^j\alpha^{\rho-1-2j}\int_{\tau^{-\gamma}}^{\infty}z^{\frac{(\rho-1)(\gamma-1)+j}{\gamma}+i}e^{-\rho z}dz\Biggr\}\\
		&=\gamma^{\rho-1}\sum_{i=0}^{\infty}\sum_{j=0}^{\infty}\binom{\rho}{i}\binom{\frac{\rho-1}{2}}{j}p^i(1-p)^{\rho-i}(2\beta)^j\alpha^{\rho-1-2j}\int_{0}^{\infty} z^{\frac{(\rho-1)(\gamma-1)+j}{\gamma}+i}e^{-\rho z}dz\\
		&=\gamma^{\rho-1}\sum_{i=0}^{\infty}\sum_{j=0}^{\infty}\binom{\rho}{i}\binom{\frac{\rho-1}{2}}{j}p^i(1-p)^{\rho-i}(2\beta)^j\alpha^{\rho-1-2j}\frac{\Gamma\left(\frac{(\rho-1)(\gamma-1)+j}{\gamma}+i+1\right) }{\rho^{\frac{(\rho-1)(\gamma-1)+j}{\gamma}+i+1}}.\numberthis\label{r1}
	\end{align*}
Now, by using \eqref{r1} in \eqref{r}, we get desired result.
	\section{Parameters Estimation}\label{sec5}
	This section deals with the estimation of unknown parameters of RTGLE $(\alpha,\beta,\gamma,p)$ distribution. Various methods of estimation have applied to provide the estimates, and the performance of these methods have assessed via a simulation study.  
	\subsection{Maximum Likelihood Estimators}
	Let $X_{1}, X_{2}, \ldots, X_{n}$ be a random sample taken from RTGLE distribution with four parameters $\alpha,\beta, \gamma$ and $p$. The log-likelihood function is given by
	\begin{align*}
		l(\alpha,\beta, \gamma,p|\boldsymbol{x})&=n\log \gamma+\sum_{i=1}^{n} \log \left(\alpha+\beta x_{i}\right)+(\gamma-1) \sum_{i=1}^{n} \log \left(\alpha x_{i}+\frac{\beta x_{i}^{2}}{2}\right)\\&\quad+\sum_{i=1}^{n} \log \left[(1-p)+p\left(\alpha x_{i}+\frac{\beta x_{i}^{2}}{2}\right)^{\gamma}\right]-\sum_{i=1}^{n}\left(\alpha x_{i}+\frac{\beta x_{i}^{2}}{2}\right)^{\gamma},
	\end{align*}
	where $\boldsymbol{x}=(x_1,x_2,\ldots,x_n)$. The maximum likelihood estimators (MLEs) of parameters are obtained by solving the following  non-linear equations.
	\begin{align}\label{MLEq1}
		\frac{\partial}{\partial \alpha}l(\alpha,\beta, \gamma,p|\boldsymbol{x})&=\sum_{i=1}^{n} \frac{1}{\alpha+\beta x_{i}}+\sum_{i=1}^{n} \frac{2(\gamma-1)x_i }{2\alpha x_{i}+\beta x_{i}^{2}}+p\gamma\sum_{i=1}^{n} \frac{ x_{i}\left(\alpha x_{i}+\frac{\beta x_{i}^{2}}{2}\right)^{\gamma-1}}{1-p+p\left(\alpha x_{i}+\frac{\beta x_{i}^{2}}{2}\right)^{\gamma}}
		\nonumber\\& \qquad\qquad-\gamma\sum_{i=1}^{n} x_{i}\left(\alpha x_{i}+\frac{\beta x_{i}^{2}}{2}\right)^{\gamma-1}=0,
	\end{align}
	\begin{align}\label{MLEq2}
		\frac{\partial}{\partial \beta}l(\alpha,\beta, \gamma,p|\boldsymbol{x})&=
		\sum_{i=1}^{n} \frac{x_i}{\alpha+\beta x_{i}}+\sum_{i=1}^{n} \frac{(\gamma-1)x_i^2 }{2\alpha x_{i}^2+\beta x_{i}^{2}}+
		\frac{p\gamma}{2}\sum_{i=1}^{n} \frac{ x_{i}\left(\alpha x_{i}+\frac{\beta x_{i}^{2}}{2}\right)^{\gamma-1}}{ 1-p+p\left(\alpha x_{i}+\frac{\beta x_{i}^{2}}{2}\right)^{\gamma} }
		\nonumber\\&\qquad \qquad -\frac{\gamma}{2}\sum_{i=1}^{n} x_{i}^2\left(\alpha x_{i}+\frac{\beta x_{i}^{2}}{2}\right)^{\gamma-1}=0,
	\end{align}
	\begin{align}\label{MLEq3}
		\frac{\partial}{\partial \gamma}l(\alpha,\beta, \gamma,p|\boldsymbol{x})&=-\frac{n}{\gamma}+\sum_{i=1}^{n} \log \left(\alpha x_{i}+\frac{\beta x_{i}^{2}}{2}\right)+p\sum_{i=1}^{n}\frac{\left(\alpha x_{i}+\frac{\beta x_{i}^{2}}{2}\right)^{\gamma}\log \left(\alpha x_{i}+\frac{\beta x_{i}^{2}}{2}\right)}{1-p+p\left(\alpha x_{i}+\frac{\beta x_{i}^{2}}{2}\right)^{\gamma} }\nonumber\\&\qquad \qquad
		-\sum_{i=1}^{n} \left(\alpha x_{i}+\frac{\beta x_{i}^{2}}{2}\right)^{\gamma}\log \left(\alpha x_{i}+\frac{\beta x_{i}^{2}}{2}\right)=0,
	\end{align}
	and
	\begin{align}\label{MLEq4}
		\frac{\partial}{\partial p}l(\alpha,\beta, \gamma,p|\boldsymbol{x})=\sum_{i=1}^{n} \frac{\left(\alpha x_{i}+\frac{\beta x_{i}^{2}}{2}\right)^{\gamma}-1}{1-p+p\left(\alpha x_{i}+\frac{\beta x_{i}^{2}}{2}\right)^{\gamma}}=0.
	\end{align}
The solution of the non-linear equations (\ref{MLEq1})-(\ref{MLEq4}) cannot be obtained in closed form, so we will use the Monte Carlo simulation to obtain the MLEs of the unknown parameters $\alpha, \beta, \gamma$ and $p$. 

\subsection{Least Squares and Weighted Least Squares Estimators}
Swain et al. (1988) suggested the method of least squares estimation to obtain the estimates of unknown parameters by minimizing the distance between the vector of uniformized order statistics and the corresponding vector of expected values. This method can be described as: Let $X_{1}, X_{2}, \ldots, X_{n}$ be a random sample of size $n$ from a continuous probability distribution with DF $F(\cdot)$ and let $X_{(1)} \leq X_{(2)} \leq \cdots \leq X_{(n)}$ denote the corresponding  order statistics. Then, for $1 \leq i \leq n$, the transformed variable $F(X_{(i)})$ has the distribution of $i$th uniform order statistic, i.e., the $i$th smallest observation in a random sample of size $n$ from the standard uniform distribution over $(0,1)$. The mean and variance of the transformed variable $F(X_{(i)})$ are, respectively, given by  
	\begin{align*}
		E\left(F\left(X_{(i)}\right)\right)=\frac{i}{n+1} ; i=1,2, \ldots, n,
	\end{align*}
	and
	\begin{align*}
		\operatorname{Var}\left(F\left(X_{(i)}\right)\right)=\frac{i(n-i+1)}{(n+1)^{2}(n+2)} ; i=1,2, \ldots, n.
	\end{align*}
Using these values, two methods of least squares can be used.\\
{\bf Least Squares Estimators:}
The least square estimators (LSEs), $\hat{\alpha}_{LSE}, \hat{\beta}_{LSE}$, $\hat{\gamma}_{LSE}$ and $\hat{p}_{LSE}$, of $\alpha, \beta,\gamma$ and $p$ can be obtained by minimizing 
	\begin{align}\label{eq4.7}
		S(\alpha, \beta,\gamma, p)=\sum_{i=1}^{n}\left(F\left(x_{(i)} \right)-\frac{i}{n+1}\right)^{2},
	\end{align}
with respect to the unknown parameters $\alpha, \beta,\gamma$ and $p$. Now, using the DF given in equation (\ref{eq2.3}) and equation (\ref{eq4.7}), we can derive four non-linear equations after partially differentiating with respect to unknown parameters. The solution of these non-linear equations can be computed by using the Monte Carlo simulation. \\
{\bf Weighted Least Squares Estimators:} The weighted least squares estimators (WLSEs), $\hat{\alpha}_{WLSE},\hat{\beta}_{WLSE}, \hat{\gamma}_{WLSE}$ and $\hat{p}_{WLSE}$ can be obtained by minimizing 
	\begin{equation*}
		W(\alpha, \beta,\gamma, p)=\sum_{i=1}^{n} \tau_{i}\left(F\left(X_{(i)}\right)-\frac{i}{n+1}\right)^{2},
	\end{equation*}
with respect to unknown parameters, where 
$\tau_{i}=\frac{1}{\operatorname{Var}\left(F\left(X_{(i)}\right)\right)}=\frac{(n+1)^{2}(n+2)}{i(n-i+1)}$. Therefore, in case of RTGLE distribution, the WLSEs of $\alpha, \beta,\gamma$ and $p$ can be obtained by using the similar procedure to least squares estimators.     
	
	\subsection{Anderson-Darling Estimator}\label{section5.3}
	Anderson-Darling estimator (ADE) is defined as an estimator that minimizes the Anderson-Darling distance between the empirical DF and theoretical DF. For more details, see Anderson and Darling (1952), and Boos (1982).  Let $X_{1}, X_{2}, \ldots, X_{n}$ be a random sample of size $n$ from RTGLE distribution with DF $F(\cdot)$ and let $X_{(1)} \leq X_{(2)} \leq \cdots \leq X_{(n)}$ denote the corresponding  order statistics. Then, the ADEs $\hat{\alpha}_{ADE},\hat{\beta}_{ADE}, \hat{\gamma}_{ADE}$ and $\hat{p}_{ADE}$ of unknown parameters of the RTGLE distribution can be obtained by minimizing the following equation 
	\begin{equation*}
		A(\alpha, \beta,\gamma, p)=-n-\frac{1}{n} \sum_{i=1}^{n}(2 i-1)\left[\log F\left(X_{(i)}\right)+\log \bar{F}\left(X_{(n+1-i)}\right)\right],
	\end{equation*}
with respect to unknown parameters, where $\bar{F}(x)=1-F(x)$ denotes the survival function of RTGLE distribution. 
	\subsection{Cram\'er-von Mises Estimator}
This is another method of estimation based on minimum-distance between the empirical DF and theoretical DF. Similar to Anderson-Darling method, we want to minimize the Cram\'er-von Mises statistic 
	\begin{equation*}
		C(\alpha, \beta,\gamma, p)=\frac{1}{12 n}+\sum_{i=1}^{n}\left(F\left(X_{(i)}\right)-\frac{2 i-1}{2 n}\right)^{2},
	\end{equation*}
with respect to unknown parameters, where all notations are same as given in subsection \ref{section5.3}. 

	\section{Simulation}\label{sec6}
	In this section, we portray a simulation study to assess the new proposed RTGLE distribution.  The Monte Carlo algorithm is used for parameter estimation and calculating statistical properties. We have utilized five different methods to estimate the population parameters of the proposed RTGLE distribution. The BFGS (Broyden-Fletcher-Goldfarb-Shanno) algorithm, established by Broyden (1970), Fletcher (1970), Goldfarb (1970), Shanno (1970), is used in R programming to attain these estimates. Tables \ref{table4}-\ref{table7} represent the biases and MSEs of the proposed model for all the estimation methods MLEs, LSEs, WLSEs, ADEs and CMEs. The reported results are based on $5000$ replicates for different sample sizes $(n = 20, 50, 100, 200)$ and parameters combinations. Tables \ref{table4}-\ref{table7} reveal that biases and MSEs of all the parameters lie in a small interval, indicating the stable nature of the distribution. In many places, we also observed the reductions of MSEs with the increase of sample sizes, which is evident. The behaviors of the proposed distribution in Tables \ref{table4}-\ref{table7} imply that the distribution performs well in terms of bias and MSEs.  It is also observed that the proposed model is equally efficient in the different choices of parameters. Hence, simulation results establish the empirical evidence of the stability of our estimates.

	\begin{table}[H]
		\caption{Biases and MSEs of estimators for $\alpha=1.2$, $\beta=0.5$, $\gamma=1.5$, $p=0.8$}
		\centering
		\scalebox{0.85}{
			\begin{tabular}{|c|c|cccc|cccc|}
				\hline
				$n$&Estimator&  & \multicolumn{1}{c}{Bias}      &       &       &  &     \multicolumn{1}{c}{MSE}  &      &   \\
				\cline{1-6}\cline{7-10}
				&& \multicolumn{1}{c}{$\hat{\alpha}$} & \multicolumn{1}{c}{$\hat{\beta}$} & \multicolumn{1}{c}{$\hat{\gamma}$} & \multicolumn{1}{c|}{$\hat{p}$} & \multicolumn{1}{c}{$\hat{\alpha}$} & \multicolumn{1}{c}{$\hat{\beta}$} & \multicolumn{1}{c}{$\hat{\gamma}$} & \multicolumn{1}{c|}{$\hat{p}$} \\
				\cline{3-10}
				20&MLE & -0.1013 & 0.0394 & 0.1885 & -0.1230 & 0.1666 & 0.4020 & 0.1823 & 0.1179 \\ 
				&OLS & -0.1222 & 0.2091 & -0.0305 & -0.0897 & 0.2493 & 0.6318 & 0.1275 & 0.0744 \\ 
				&	WLS & -0.1252 & 0.2823 & -0.0438 & -0.0673 & 0.2604 & 0.8241 & 0.1349 & 0.0622 \\ 
				&	CME & -0.1090 & 0.1189 & 0.1048 & -0.0857 & 0.2263 & 0.4474 & 0.1624 & 0.0816 \\ 
				&	ADE & -0.1542 & 0.3869 & -0.0528 & -0.0419 & 0.2946 & 1.0628 & 0.1162 & 0.0623 \\ 
				\hline				
				50&MLE & -0.1330 & 0.1528 & 0.0912 & -0.1176 & 0.1619 & 0.5542 & 0.1230 & 0.1040 \\ 
				&OLS & -0.1343 & 0.1942 & -0.0223 & -0.0866 & 0.2076 & 0.4905 & 0.0681 & 0.0672 \\ 
				&	WLS & -0.1445 & 0.2627 & -0.0334 & -0.0720 & 0.2126 & 0.6633 & 0.0798 & 0.0624 \\ 
				&	CME & -0.1211 & 0.1437 & 0.0350 & -0.0859 & 0.1923 & 0.4010 & 0.0726 & 0.0739 \\ 
				&	ADE & -0.1876 & 0.4321 & -0.0769 & -0.0493 & 0.2573 & 1.0306 & 0.0887 & 0.0596 \\ 
				\hline				
				100&MLE & -0.1469 & 0.2670 & 0.0146 & -0.0867 & 0.1541 & 0.6649 & 0.1035 & 0.0863 \\ 
				&OLS & -0.1352 & 0.1950 & -0.0155 & -0.0839 & 0.1756 & 0.4548 & 0.0602 & 0.0627 \\ 
				&	WLS & -0.1477 & 0.2774 & -0.0324 & -0.0687 & 0.1799 & 0.6508 & 0.0782 & 0.0597 \\ 
				&	CME & -0.1254 & 0.1635 & 0.0151 & -0.0843 & 0.1657 & 0.3964 & 0.0629 & 0.0683 \\ 
				&	ADE & -0.2113 & 0.4870 & -0.0937 & -0.0525 & 0.2375 & 1.0724 & 0.0890 & 0.0584 \\ 
				\hline				
				200&MLE & -0.1872 & 0.3383 & 0.0106 & -0.1067 & 0.1559 & 0.8338 & 0.1203 & 0.0865 \\ 
				&OLS & -0.1570 & 0.2150 & 0.0024 & -0.1004 & 0.1565 & 0.4835 & 0.0675 & 0.0638 \\ 
				&	WLS & -0.1574 & 0.2795 & -0.0137 & -0.0809 & 0.1520 & 0.6417 & 0.0867 & 0.0603 \\ 
				&	CME & -0.1455 & 0.1806 & 0.0243 & -0.1021 & 0.1457 & 0.4218 & 0.0693 & 0.0682 \\ 
				&	ADE & -0.2311 & 0.5331 & -0.0980 & -0.0576 & 0.2163 & 1.1162 & 0.0977 & 0.0556 \\ 
				\hline		
			\end{tabular}
		}
		\label{table4}
	\end{table}
	\begin{table}[H]
		\caption{Biases and MSEs of estimators for $\alpha=1.1$, $\beta=1.2$, $\gamma=2$, $p=0.8$}
		\centering
		\scalebox{0.85}{
			\begin{tabular}{|c|c|cccc|cccc|}
				\hline
				$n$&Estimator&  & \multicolumn{1}{c}{Bias}      &       &       &  &     \multicolumn{1}{c}{MSE}  &      &   \\
				\cline{1-6}\cline{7-10}
				&& \multicolumn{1}{c}{$\hat{\alpha}$} & \multicolumn{1}{c}{$\hat{\beta}$} & \multicolumn{1}{c}{$\hat{\gamma}$} & \multicolumn{1}{c|}{$\hat{p}$} & \multicolumn{1}{c}{$\hat{\alpha}$} & \multicolumn{1}{c}{$\hat{\beta}$} & \multicolumn{1}{c}{$\hat{\gamma}$} & \multicolumn{1}{c|}{$\hat{p}$} \\
				\cline{3-10}
				20&MLE & 0.1003 & -0.3637 & 0.3586 & -0.0460 & 0.1543 & 0.7475 & 0.4135 & 0.0827 \\ 
				&LSE & 0.0157 & -0.1255 & 0.0900 & -0.0679 & 0.2140 & 0.8335 & 0.2391 & 0.0639 \\ 
				&	WLS & 0.0262 & -0.1635 & 0.1655 & -0.0786 & 0.2084 & 1.0482 & 0.3446 & 0.0667 \\ 
				&	CME & 0.0527 & -0.2985 & 0.3247 & -0.0710 & 0.1879 & 0.7119 & 0.4079 & 0.0735 \\ 
				&	ADE & 0.0435 & -0.2015 & 0.2128 & -0.0564 & 0.2052 & 1.0706 & 0.3201 & 0.0650 \\ 
				\hline
				
				50&MLE & 0.0322 & -0.2185 & 0.2663 & -0.0861 & 0.1305 & 0.8220 & 0.3103 & 0.0893 \\ 
				&LSE & -0.0096 & -0.1061 & 0.0988 & -0.0737 & 0.1653 & 0.6696 & 0.1286 & 0.0613 \\ 
				&	WLS & -0.0035 & -0.1469 & 0.1756 & -0.0929 & 0.1719 & 0.9229 & 0.2221 & 0.0702 \\ 
				&	CME & 0.0333 & -0.2578 & 0.2256 & -0.0775 & 0.1536 & 0.6128 & 0.1895 & 0.0676 \\ 
				&	ADE & -0.0511 & 0.0585 & 0.0902 & -0.0565 & 0.1970 & 1.1469 & 0.1872 & 0.0581 \\ 
				\hline
				
				100&MLE & -0.0176 & -0.0343 & 0.1547 & -0.0700 & 0.1182 & 0.8567 & 0.2430 & 0.0762 \\ 
				&LSE & -0.0303 & -0.0630 & 0.0977 & -0.0804 & 0.1380 & 0.6098 & 0.1141 & 0.0612 \\ 
				&	WLS & -0.0214 & -0.1118 & 0.1793 & -0.1006 & 0.1497 & 0.9059 & 0.2230 & 0.0695 \\ 
				&	CME & 0.0164 & -0.2260 & 0.2013 & -0.0890 & 0.1291 & 0.5684 & 0.1559 & 0.0672 \\ 
				&	ADE & -0.0219 & -0.1308 & 0.2182 & -0.1136 & 0.1526 & 0.9843 & 0.2682 & 0.0794 \\ 
				\hline
				
				200&MLE & -0.0766 & 0.1023 & 0.1230 & -0.0921 & 0.1166 & 0.9654 & 0.2575 & 0.0778 \\ 
				&LSE & -0.0630 & 0.0058 & 0.0933 & -0.0939 & 0.1153 & 0.5569 & 0.1152 & 0.0610 \\ 
				&	WLS & -0.0452 & -0.0575 & 0.1769 & -0.1085 & 0.1312 & 0.8833 & 0.2338 & 0.0654 \\ 
				&	CME & -0.0038 & -0.2102 & 0.2189 & -0.1141 & 0.1097 & 0.5609 & 0.1737 & 0.0698 \\ 
				&	ADE & -0.0484 & -0.0623 & 0.2029 & -0.1180 & 0.1368 & 0.9777 & 0.2694 & 0.0713 \\ 
				\hline
				
			\end{tabular}
		}
		\label{table5}
	\end{table}
	\begin{table}[H]
		\caption{Biases and MSEs of estimators for $\alpha=0.8$, $\beta=0.8$, $\gamma=1$, $p=0.7$}
		\centering
		\scalebox{0.85}{
			\begin{tabular}{|c|c|cccc|cccc|}
				\hline
				$n$&Estimator&  & \multicolumn{1}{c}{Bias}      &       &       &  &     \multicolumn{1}{c}{MSE}  &      &   \\
				\cline{1-6}\cline{7-10}
				&& \multicolumn{1}{c}{$\hat{\alpha}$} & \multicolumn{1}{c}{$\hat{\beta}$} & \multicolumn{1}{c}{$\hat{\gamma}$} & \multicolumn{1}{c|}{$\hat{p}$} & \multicolumn{1}{c}{$\hat{\alpha}$} & \multicolumn{1}{c}{$\hat{\beta}$} & \multicolumn{1}{c}{$\hat{\gamma}$} & \multicolumn{1}{c|}{$\hat{p}$} \\
				\cline{3-10}
				20&MLE & 0.0230 & -0.1721 & 0.2000 & -0.1035 & 0.1665 & 0.4711 & 0.1318 & 0.1076 \\ 
				&LSE & 0.0355 & 0.0210 & 0.0078 & -0.0175 & 0.2806 & 0.6420 & 0.0609 & 0.0600 \\ 
				&	WLS & 0.0160 & 0.1420 & -0.0106 & -0.0032 & 0.2921 & 0.8093 & 0.0644 & 0.0515 \\ 
				&	CME & 0.0533 & -0.0947 & 0.1115 & -0.0181 & 0.2642 & 0.5076 & 0.0850 & 0.0710 \\ 
				&	ADE & 0.0074 & 0.1343 & 0.0227 & -0.0047 & 0.2894 & 0.8422 & 0.0664 & 0.0596 \\ 
				\hline
				
				50&MLE & -0.0497 & -0.0347 & 0.1133 & -0.1171 & 0.1294 & 0.4841 & 0.0746 & 0.0925 \\ 
				&LSE & -0.0018 & 0.0393 & 0.0033 & -0.0234 & 0.2098 & 0.5019 & 0.0353 & 0.0522 \\ 
				&	WLS & -0.0486 & 0.1459 & -0.0083 & -0.0301 & 0.1933 & 0.6118 & 0.0398 & 0.0498 \\ 
				&	CME & 0.0197 & -0.0332 & 0.0498 & -0.0238 & 0.2060 & 0.4409 & 0.0401 & 0.0591 \\ 
				&	ADE & -0.0542 & 0.1563 & 0.0005 & -0.0289 & 0.1901 & 0.6392 & 0.0416 & 0.0529 \\ 
				\hline
				
				100&MLE & -0.0871 & 0.0406 & 0.0783 & -0.1189 & 0.1001 & 0.4687 & 0.0624 & 0.0810 \\ 
				&LSE & -0.0494 & 0.0873 & -0.0005 & -0.0408 & 0.1533 & 0.4499 & 0.0308 & 0.0483 \\ 
				&	WLS & -0.0903 & 0.1608 & -0.0037 & -0.0552 & 0.1366 & 0.5176 & 0.0370 & 0.0487 \\ 
				&	CME & -0.0341 & 0.0413 & 0.0267 & -0.0429 & 0.1520 & 0.4241 & 0.0335 & 0.0534 \\ 
				&	ADE & -0.0991 & 0.1865 & -0.0041 & -0.0535 & 0.1381 & 0.5559 & 0.0390 & 0.0509 \\ 
				\hline
				
				200&MLE & -0.1106 & 0.0573 & 0.0614 & -0.1213 & 0.0784 & 0.3898 & 0.0496 & 0.0650 \\ 
				&LSE & -0.1029 & 0.1500 & -0.0044 & -0.0596 & 0.1190 & 0.4478 & 0.0306 & 0.0437 \\ 
				&	WLS & -0.1288 & 0.1866 & -0.0051 & -0.0746 & 0.1026 & 0.4466 & 0.0339 & 0.0460 \\ 
				&	CME & -0.0895 & 0.1103 & 0.0139 & -0.0630 & 0.1172 & 0.4225 & 0.0320 & 0.0476 \\ 
				&	ADE & -0.1304 & 0.2052 & -0.0085 & -0.0684 & 0.1040 & 0.4723 & 0.0343 & 0.0456 \\ 
				\hline
			\end{tabular}
		}
		\label{table6}
	\end{table}
	
	\begin{table}[H]
		\caption{Biases and MSEs of estimators for $\alpha=0.8$, $\beta=0.5$, $\gamma=0.8$, $p=0.7$}
		\centering
		\scalebox{0.85}{
			\begin{tabular}{|c|c|cccc|cccc|}
				\hline
				$n$&Estimator&  & \multicolumn{1}{c}{Bias}      &       &       &  &     \multicolumn{1}{c}{MSE}  &      &   \\
				\cline{1-6}\cline{7-10}
				&& \multicolumn{1}{c}{$\hat{\alpha}$} & \multicolumn{1}{c}{$\hat{\beta}$} & \multicolumn{1}{c}{$\hat{\gamma}$} & \multicolumn{1}{c|}{$\hat{p}$} & \multicolumn{1}{c}{$\hat{\alpha}$} & \multicolumn{1}{c}{$\hat{\beta}$} & \multicolumn{1}{c}{$\hat{\gamma}$} & \multicolumn{1}{c|}{$\hat{p}$} \\
				\cline{3-10}
				20&MLE & -0.1790 & -0.0853 & 0.1946 & -0.2534 & 0.1040 & 0.2485 & 0.1050 & 0.1601 \\ 
				&LSE & -0.1118 & 0.0075 & 0.0632 & -0.1612 & 0.1677 & 0.3953 & 0.0564 & 0.1112 \\ 
				&	WLS & -0.1257 & 0.1079 & 0.0355 & -0.1360 & 0.1823 & 0.5404 & 0.0535 & 0.1031 \\ 
				&	CME & -0.1047 & -0.0811 & 0.1535 & -0.1717 & 0.1488 & 0.2960 & 0.0818 & 0.1244 \\ 
				&	ADE & -0.1215 & 0.0951 & 0.0592 & -0.1289 & 0.2096 & 0.5143 & 0.0560 & 0.1073 \\ 
				\hline
				50&MLE & -0.2073 & -0.0105 & 0.1240 & -0.2441 & 0.1085 & 0.2855 & 0.0610 & 0.1367 \\ 
				&LSE & -0.0618 & 0.0208 & 0.0234 & -0.0841 & 0.1937 & 0.2767 & 0.0272 & 0.0660 \\ 
				&	WLS & -0.1096 & 0.0837 & 0.0239 & -0.1045 & 0.1679 & 0.3482 & 0.0318 & 0.0717 \\ 
				&	CME & -0.0495 & -0.0367 & 0.0633 & -0.0913 & 0.1863 & 0.2312 & 0.0311 & 0.0752 \\ 
				&	ADE & -0.1404 & 0.0857 & 0.0386 & -0.1282 & 0.1587 & 0.3569 & 0.0356 & 0.0878 \\ 
				\hline
				100&MLE & -0.2522 & -0.0007 & 0.1100 & -0.2795 & 0.1121 & 0.2427 & 0.0548 & 0.1427 \\ 
				&LSE & -0.1333 & 0.0198 & 0.0456 & -0.1385 & 0.1237 & 0.2386 & 0.0313 & 0.0753 \\ 
				&	WLS & -0.1886 & 0.0598 & 0.0529 & -0.1792 & 0.1081 & 0.2667 & 0.0386 & 0.0933 \\ 
				&	CME & -0.1231 & -0.0216 & 0.0716 & -0.1467 & 0.1193 & 0.2138 & 0.0344 & 0.0815 \\ 
				&	ADE & -0.1708 & 0.0850 & 0.0362 & -0.1474 & 0.1205 & 0.2815 & 0.0320 & 0.0848 \\ 
				\hline
				200&MLE & -0.1970 & 0.0435 & 0.0649 & -0.1923 & 0.0936 & 0.2284 & 0.0336 & 0.0887 \\ 
				&LSE & -0.1368 & 0.0824 & 0.0194 & -0.1052 & 0.1099 & 0.2472 & 0.0233 & 0.0523 \\ 
				&	WLS & -0.1675 & 0.1162 & 0.0177 & -0.1241 & 0.1026 & 0.2610 & 0.0250 & 0.0590 \\ 
				&	CME & -0.1285 & 0.0501 & 0.0353 & -0.1119 & 0.1073 & 0.2238 & 0.0243 & 0.0569 \\ 
				&	ADE & -0.1669 & 0.1240 & 0.0157 & -0.1201 & 0.1020 & 0.2682 & 0.0247 & 0.0615 \\ 
				\hline
				
			\end{tabular}
		}
		\label{table7}
	\end{table}

	\section{Real Data Analysis}\label{sec7}
In this section, we have utilized the real data set to see the proficiency of the proposed distribution. This section consists of two parts; in the first part, we have done the exploratory data analysis and calculate MLEs of the proposed distribution with seven other competing models. In the second part, we have done a comparison of the proposed distribution with other models. We use the  performance validation criteria such as log-likelihood, Akaike's Information Criterion (AIC) and few goodness-of-fit tests statistics Kolmogorov-Smirnov (KS), Cram\'er-von Mises (CVM) and Anderson-Darling (AD) for comparison of the proposed RTGLE distribution with other seven distributions.

  To carry out the real data analysis, we have utilized failure times (in weeks) data set by Murthy at el. (2004, p. 180). \\
  {\bf Failure Time Data:} The dataset consists of 50 observations and given by as follows; 0.013, 0.065, 0.111, 0.111, 0.163, 0.309, 0.426, 0.535, 0.684, 0.747, 0.997, 1.284, 1.304, 1.647, 1.829, 2.336, 2.838, 3.269, 3.977, 3.981, 4.520, 4.789, 4.849, 5.202, 5.291, 5.349, 5.911, 6.018, 6.427, 6.456, 6.572, 7.023, 7.087, 7.291, 7.787, 8.596, 9.388, 10.261, 10.713, 11.658, 13.006, 13.388, 13.842, 17.152, 17.283, 19.418, 23.471, 24.777, 32.795, 48.105.
	
	The data is fitted to RTGLE distribution, RTW distribution (see, Tanis and Sara\c{c}o\u{g}lu (2020)), Weibull (W) distribution, Transmuted Weibull (TW) distribution (see, Aryal and Tsokos (2011)), Transmuted Lindley (TL) distribution (see, Merovci (2013)) and Transmuted Log-Logistic (TLL) distribution (see, Aryal (2013)), RT-Linear exponential (RTLE) distribution (see, Balakrishnan and He (2021)), Linear exponential (LE) distribution (see, Bain (1974)). The PDFs of these distributions are given as follows:
	\begin{equation*}
		\begin{gathered}
			f_{RTW}(x)=\theta \gamma x^{\gamma-1} e^{-\theta x^{\gamma}}\left[1+p\left(\theta x^{\gamma}-1\right)\right], ~x,\theta,\gamma>0, p\in[0,1]\\
			f_{W}(x)=\frac{\mu}{\sigma}\left(\frac{x}{\sigma}\right)^{\mu-1} e^{-\left(\frac{x}{\sigma}\right)^{\mu}}, ~x, \mu, \sigma>0, \\
			f_{TW}(x)=\frac{\mu}{\sigma}\left(\frac{x}{\sigma}\right)^{\mu-1} e^{-\left(\frac{x}{\sigma}\right)^{\mu}}\left[1-\lambda+2 \lambda e^{-\left(\frac{x}{\sigma}\right)^{\mu}}\right], ~x, \mu, \sigma>0, \lambda \in[-1,1], \\
			f_{TL}(x)=\frac{\theta^{2}}{\theta+1}(1+x) e^{-\theta x}\left[1-\lambda+2 \lambda \frac{\theta+1+\theta x}{\theta+1} e^{-\theta x}\right], ~x, \theta>0, \lambda \in[-1,1]\\
			f_{TLL}(x)=\frac{\beta \alpha^{\beta} x^{\beta-1}\left[\left(\alpha^{\beta}+x^{\beta}\right)-2 \lambda x^{\beta}\right]}{\left(\alpha^{\beta}+x^{\beta}\right)^{3}}, ~x, \alpha, \beta>0, \lambda \in[-1,1]\\
			f_{RTLE}(x) =(\alpha+\beta x)
			\left[1-p+p\left(\alpha x+\frac{\beta x^{2}}{2}\right)\right]
			e^{-\left(\alpha x+\frac{\beta x^{2}}{2}\right)}, ~x,\alpha,\beta>0, p\in[0,1]	
		\end{gathered}
	\end{equation*}
	and
	\begin{equation*}
		f_{LE}(x) =(\alpha+\beta x)
		\left(\alpha x+\frac{\beta x^{2}}{2}\right)
		e^{-\left(\alpha x+\frac{\beta x^{2}}{2}\right)}, ~x,\alpha,\beta>0.
	\end{equation*}
	In the exploratory data analysis part, we have done outlier analysis, which helps to identify outliers in our dataset. This analysis may improve the results as it eliminates inaccurate observations, leading to biased conclusions. One of the best methods for detecting outliers is Boxplot; therefore, we drew the Boxplot and removed the outliers in the data set for further analysis. Boxplot, fitted PDFs plots and empirical and theoretical DFs plots for fitted distributions are shown in Figures \ref{f4}-\ref{f5}. These figures have evidenced a good fit of the real data set with RTGLE distribution. Table \ref{table8} reports the MLEs for all the models mentioned above, along with their standard errors. To show the effectiveness of RTGLE distribution, we have calculated three different goodness-of-tests statistics, namely, Kolmogorov-Smirnov (KS), Cramer von Mises (CVM) and Anderson-Darling (AD) with the corresponding $p$-values. Further, the values of $-2 \log L$ and AIC$=2r-2 \log L$ have been evaluated, where $L$ denotes the maximized value of the likelihood function, and $r$ is the number of parameters. The results are presented in Table \ref{table9}, and suggest that the RTGLE distribution is the best model among other competing lifetime models. Therefore, the RTGLE distribution may be a good alternative to other lifetime models having the similar statistical properties. 
	\begin{figure}[H]
		\caption{Fitted PDFs}
		\centering
		\label{f4}
		\includegraphics[width=1\linewidth, height=9cm]{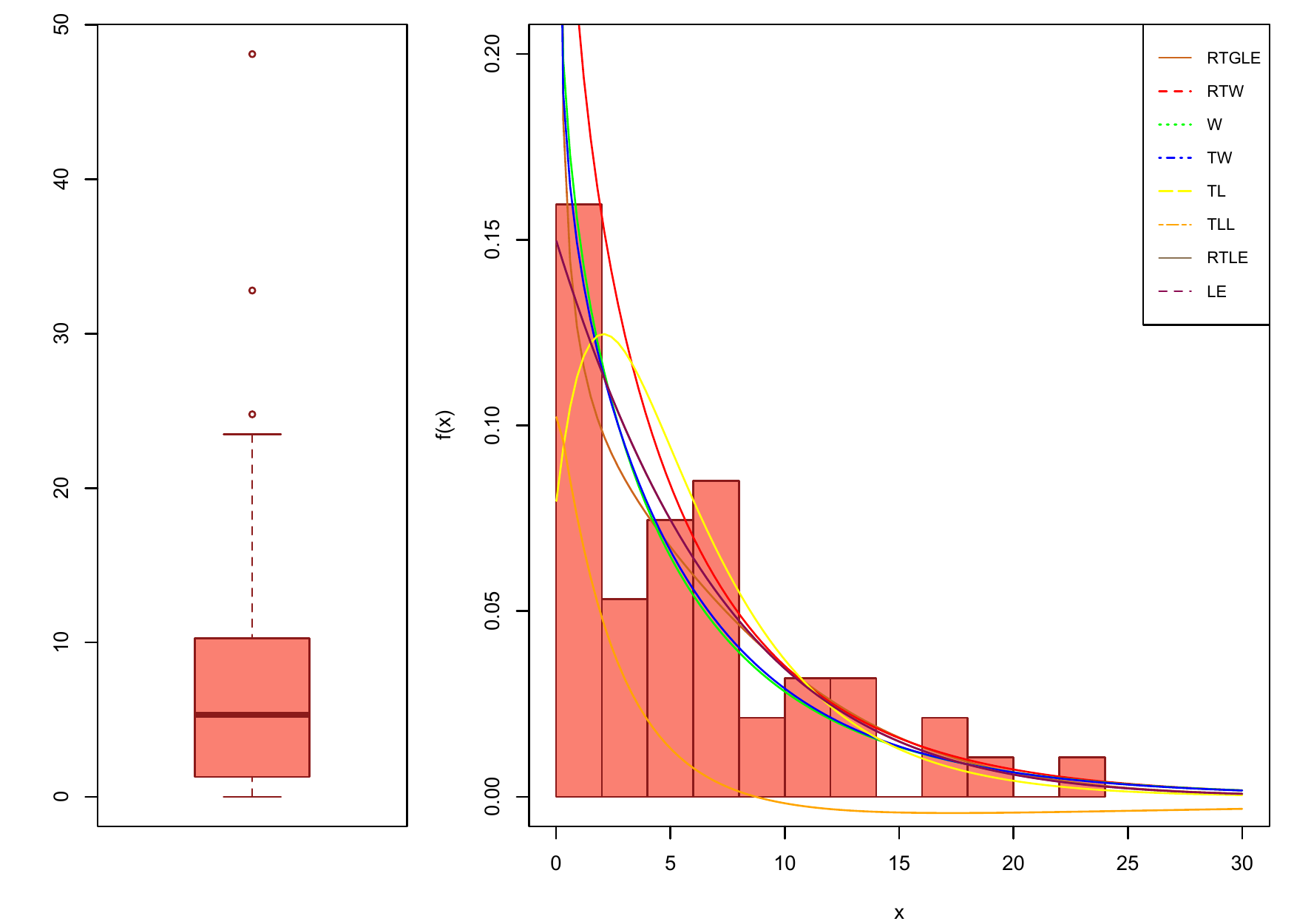} 
	\end{figure}  

\begin{table}[H]
	\caption{MLEs (standard errors) for failure time data }
	\centering
		\scalebox{0.89}{
	\begin{tabular}{|l|l|}
		\hline
		\multicolumn{1}{|c|}{Distribution} & 
		\multicolumn{1}{|c|}{MLEs(standard errors)}\\
		\hline
		RTGLE&$\hat{\alpha}=0.1561~(0.1522)$, $\hat{\beta}=0.0411~(0.0430)$, $\hat{\gamma}=0.6199~(0.1561)$, $\hat{p}=0.4068~(0.4576)$\\
		RTW&$\hat{\theta}=0.3714~(0.1650)$, $\hat{\gamma}=0.8096~(0.1317)$,  $\hat{p}=0.4901~(0.2893)$\\
		W&$\hat{\mu}=0.8857~(0.1077)$, $\hat{\sigma}=5.7710~(0.9938)$\\
		TW&$\hat{\mu}=0.8316~(0.1320)$, $\hat{\sigma}=4.8119~(1.425)$, $\hat{\lambda}=-0.2898~(0.3729)$\\
		TL&$\hat{\theta}=0.2653~(0.0466)$, $\hat{\lambda}=0.2739~(0.3543)$\\
		TLL&$\hat{\alpha}=8.7446~(1.8384)$, $\hat{\beta}=1.0198~(0.116)$, $\hat{\lambda}=1~(0.4225)$\\
		RTLE&$\hat{\alpha}=0.1628~(0.1286)$, $\hat{\beta}=0.0027~(0.0051)$, $\hat{p}=0.0797~(0.7240)$\\
		LE&$\hat{\alpha}=0.1499~(0.0341)$, $\hat{\beta}=0.0026~(0.0047)$,  \\
		\hline
	\end{tabular}
}
\label{table8}
\end{table}

\begin{table}[H]
	\caption{Selection criteria statistics for failure time data}
	\centering
	\scalebox{0.93}{
	\begin{tabular}{|l|cccccccc|}
		\hline
		Model& $-2\log L$ & AIC & KS & p(KS) & CVM &  p(CVM)& AD & p(AD) \\
		\hline
		\textbf{RTGLE} & $\boldsymbol{258.2350}$ & $\boldsymbol{266.2350}$ & $\boldsymbol{0.0869}$ & $\boldsymbol{0.8698}$ & $\boldsymbol{0.0503}$ & $\boldsymbol{0.8768}$ & $\boldsymbol{0.2836}$ & $\boldsymbol{0.9497}$ \\
		RTW & 261.4522 & 267.4522 & 0.1170 & 0.5403 & 0.1060 & 0.5579 & 0.6295 & 0.6197 \\
		W & 262.4914 & 266.4914 & 0.1298 & 0.4065 & 0.1303 & 0.4569 & 0.7487 & 0.5187  \\
		TW & 261.9948 & 267.9948 & 0.1202 & 0.5054 & 0.1132 & 0.5258 & 0.6758 & 0.5785 \\
		TL & 270.8568 & 274.8568 & 0.1531 & 0.2203 & 0.2023 & 0.2636 & 2.4818 & 0.0509  \\
		TLL & 269.4633 & 275.4633 & 0.1406 & 0.3106 & 0.1704 & 0.3338 & 1.1554 & 0.2849\\
		RTLE & 263.2147 & 269.2147 & 0.1062 & 0.6637 & 0.1065 & 0.5559 & 1.0587 & 0.3275\\
		LE & 263.2192 & 267.2192 & 0.1062 & 0.6647 & 0.1067 & 0.5550 & 1.0585 & 0.3276 \\
		\hline
	\end{tabular}
}
\label{table9}
\end{table}
	\begin{figure}[H]
		\caption{Fitted CDFs}
		\centering
		\label{f5}
		\includegraphics[width=1\linewidth, height=10cm]{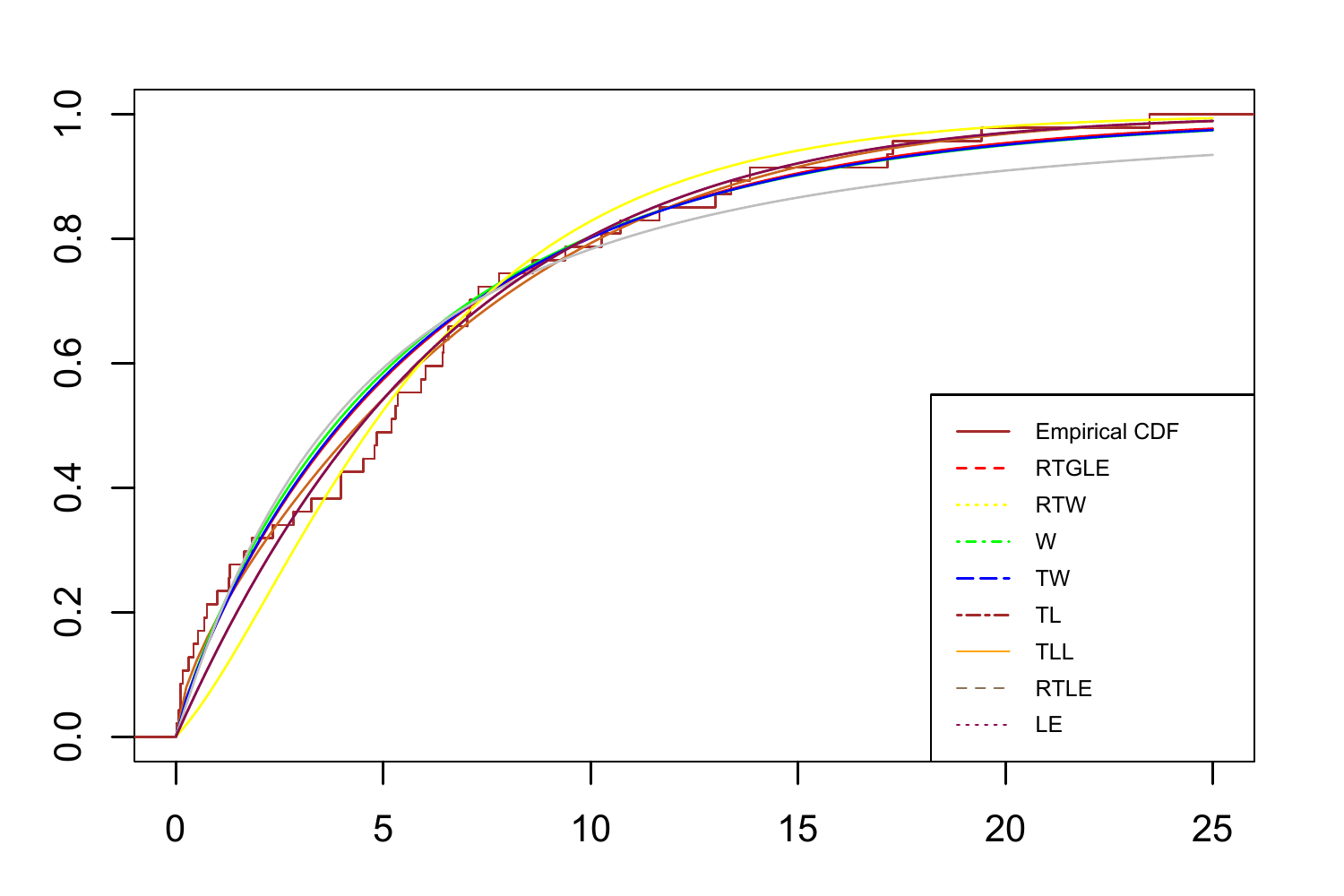} 
	\end{figure}
\section{Conculding Remarks}\label{sec8}
This paper investigated the proposed RT-generalized linear exponential (RTGLE) distribution and its properties IFR, DFR and BTFR. We have also discussed various statistical properties and derived different estimators using multiple estimation methods of the proposed distribution. A rigorous simulation study has been carried out before examining real data to estimate and infer the proposed distribution. The simulation results establish that an MC algorithm works well for the distribution and provides small MSEs, even for the small sample sizes. Further, the significance of the proposed distribution is tested by analyzing the lifetime data set and the results show the dominant nature of the proposed distribution. From the above analyses, we may conclude that the proposed distribution is beneficial in real-life scenarios. The results of the proposed RTGLE distribution showed its flexibility and effectiveness for the artificial and real population as well. Hence encourages its usage in the future and can recommend them for their practical applications in real-life problems. We have used R programming to perform all the simulations.


\begin{thebibliography}{99}
	
	\bibitem{Ahsanullah1995}
	Ahsanullah M. (1995). Record Statistics. New York, Nova Science Publishers,
	Inc. Commack.
	
	\bibitem{Al-Babtain2017} Al-Babtain A, Fattah AA, Ahmed AHN, Merovci F. (2017). The Kumaraswamy-transmuted exponentiated modified Weibull distribution. Communications in Statistics-Simulation	and Computation, 46 (5), 3812-3832.
	
	\bibitem{Al-Babtain2020}	
	Al-Babtain AA, Elbatal I, Chesneau C, Jamal F. (2020). The Transmuted Muth Generated Class of Distributions with Applications. Symmetry, 12 (10), 1677.
	
	\bibitem{Ahmad2020}
	Ahmad AEA, Ghazal MGM. (2020). Exponentiated additive Weibull distribution.
	Reliability Engineering \& System Safety, 193, 106663.
	
	\bibitem{Almalki2013}
	Almalki SJ, Yuan J. (2013).	A new modified Weibull distribution,
	Reliability Engineering \& System Safety, 111, 164-170.
	
	\bibitem{Almalki2014}
	Almalki SJ, Nadarajah S. (2014). Modifications of the Weibull distribution: A review. Reliability Engineering \& System Safety, 124, 32-55.
	
	\bibitem{Anderson1952}
	Anderson TW, Darling DA. (1952). Asymptotic theory of certain ``goodness-of-fit'' criteria based on stochastic processes. The Annals of Mathematical Statistics, 23 (2), 193-212.
	
	\bibitem{Arnold1992}
	 Arnold BC, Balakrishnan N, Nagaraja HN. (1992). A first course in order statistics. New York, John Wiley \& Sons.
	
	\bibitem{Arnold1998} 
	Arnold BC, Balakrishnan N, Nagaraja HN. (1998). Records. New York, John Wiley \& Sons.
	
	\bibitem{Arnold2006}
	Arnold BC, Castillo E, Sarabia JM. (2006). Families of multivariate distributions involving the Rosenblatt construction. Journal of the American Statistical Association, 101 (476), 1652-1662.
	
	\bibitem{Aryal2011} 
	Aryal GR, Tsokos CP. (2011). Transmuted Weibull distribution: A generalization of the Weibull probability distribution. European Journal of Pure and Applied Mathematics, 4 (1), 89-102.
	
	\bibitem{Aryal2013}
	Aryal GR. (2013). Transmuted log-logistic distribution. Journal of Statistics Applications \& Probability, 2 (1), 11-20. 
	
	\bibitem{Arshad2019} 
	Arshad M, Baklizi A. (2019). Estimation of common location parameter of two exponential	populations based on records. Communications in Statistics-Theory and Methods, 48(6), 1545-1552.
	
	\bibitem{Arshad2019a} 
	Arshad M, Jamal QA. (2019a) Interval Estimation for Topp-Leone
	Generated Family of Distributions based on Dual Generalized Order Statistics. American Journal of	Mathematical and Management Sciences, 38 (3), 227-240.
	
	\bibitem{Arshad2019b} 
	Arshad M, Jamal QA. (2019b). Statistical Inference for Topp-Leone-generated Family of Distributions Based on Records. Journal of Statistical Theory and Applications, 18 (1), 65-78.
	
	\bibitem{Arshad2021} 
	Arshad M, Azhad QJ, Gupta N, Pathak AK. (2021). Bayesian inference of Unit Gompertz distribution based on dual generalized order statistics. Communications in Statistics-Simulation and Computation, DOI: 10.1080/03610918.2021.1943441.
	
	\bibitem{Arshad2021} 
	Azhad QJ, Arshad M, Misra AK. (2021). Estimation of common location parameter of several heterogeneous exponential populations based
	on generalized order statistics. Journal of Applied Statistics, 48 (10), 1798-1815.
	
	\bibitem{Azzalini2013}
	Azzalini A. (2013). The skew-normal and related families. Cambridge, England: Cambridge University Press.
	
	
	\bibitem{Bain1974}
	Bain LJ. (1974). Analysis for the linear failure-rate life-testing distribution. Technometrics, 16 (4), 551-559.
	
	
	\bibitem{Balakrishnan2016}
	Balakrishnan N, Risti\'c MM. (2016). Multivariate families of gamma-generated distributions with finite or infinite support above or below the diagonal. Journal of Multivariate	Analysis, 143, 194-207.
	
	\bibitem{Balakrishnan2021}
	Balakrishnan N, He M. (2021). A Record-Based Transmuted Family of Distributions; In Advances in Statistics-Theory and Applications. 
	Emerging Topics in Statistics and Biostatistics; Ghosh I, Balakrishnan N, Ng HKT (Eds). Cham, Switzerland, Springer.
	
	\bibitem{Boos1982}
	 Boos DD. (1982). Minimum Anderson-Darling estimation.  Communications in Statistics-Theory and Methods, 11 (24), 2747-2774.
	
	
	\bibitem{Broyden1970}
	 Broyden CG. (1970). The convergence of a class of double-rank minimization algorithms 1. general considerations. IMA Journal of Applied Mathematics, 6 (1), 76-90.
	
	\bibitem{Chandler1952}
     Chandler KN. (1952). The distribution and frequency of record values. Journal of the Royal Statistical Society. Series B (Methodological), 14 (2), 220-228.
     
	\bibitem{cordeiro2011}
	Cordeiro GM, De Castro M. (2011). A new family of generalized distributions. Journal of Statistical Computation and Simulation, 81 (7), 883–898.
	
	\bibitem{Corless1996}
	Corless RM, Gonnet GH, Hare DEG, Jeffrey DJ, Knuth DE. (1996). On the Lambert W function. Adv Comput Math, 5, 329–359.
	
	\bibitem{David2003}
	David HA, Nagaraja HN. (2003). Order Statistics. Third edition. New York, John Wiley.
	
	\bibitem{Deshpande1986}
	Deshpande JV, Kochar SC, Singh H. (1986). Aspects of positive ageing. Journal of Applied Probability, 23, 748-758.
	
	
	\bibitem{Elbatal2013}
	Elbatal I, Aryal G. (2013). On the transmuted additive Weibull distribution. Austrian Journal of Statistics, 42 (2), 117-132.
	
	\bibitem{Eugene2002}	
	Eugene N, Lee C, Famoye F. (2002). Beta-normal distribution and its applications. Communications in Statistics-Theory and Methods, 31 (4), 497-512.
	
	\bibitem{Ferreira2006}
	Ferreira JTS, Steel MFJ. (2006). A constructive representation of univariate skewed distributions. Journal of the American Statistical Association, 101 (474), 823-829.
	
	\bibitem{Ferreira2007}
	Ferreira JTS, Steel MFJ. (2007). Model comparison of coordinate-free multivariate skewed distributions with an application to stochastic frontiers. Journal of Econometrics, 137 (2), 641-673.
	
	\bibitem{Fletcher1970}
	Fletcher R. (1970). A new approach to variable metric algorithms. The Computer Journal, 13 (3), 317-322.
	
	
	\bibitem{Ghosh2021}
	Ghosh S, Kataria KK, Vellaisamy P. (2021). On transmuted generalized linear exponential distribution. Communications in Statistics-Theory and Methods, 50 (9), 1978-2000.
	
	\bibitem{Glaser1980}
	Glaser RE. (1980). Bathtub and related failure rate characterizations. Journal of the American Statistical Association, 75 (371), 667-672.
	
	\bibitem{Goldfarb1970} 
	Goldfarb D. (1970). A family of variable-metric methods derived by variational means. Mathematics of Computation, 24 (109), 23-26. 
	
	%
	
	\bibitem{Granzotto2017}
	Granzotto DCT, Louzada F, Balakrishnan N. (2017). Cubic rank transmuted distributions: Inferential issues and applications. Journal of Statistical Computation and Simulation, 87 (14), 2760-2778. 
	
	\bibitem{Gupta1999}
	Gupta RD, Kundu D. (1999). Theory \& methods: Generalized exponential distributions. Australian \& New Zealand Journal of Statistics, 41 (2), 173-188.
	
	\bibitem{He2016}
	He B, Cui W, Du X. (2016). An additive modified Weibull distribution,
	Reliability Engineering \& System Safety, 145, 28-37.
	
	
	\bibitem{Pedro2021}
	Jodr\'a P, Arshad M. (2021). An intermediate muth distribution with increasing failure rate, Communications in Statistics-Theory and Methods. DOI: 10.1080/03610926.2021.1892133.
	
	\bibitem{Jones2009}
	Jones MC. (2009). Kumaraswamy’s distribution: A beta-type distribution with some tractability advantages. Statistical Methodology, 6 (1), 70-81.
	
	\bibitem{karakaya2017}
	Karakaya K, Kinaci I, Kus C, Akdogan Y. (2017). A new family of distributions. Hacettepe Journal of Mathematics and Statistics, 46 (2), 303-314.
	
	\bibitem{Khan2014} 
	Khan MS, King R. (2014). A new class of transmuted inverse Weibull distribution for reliability analysis. American Journal of Mathematical and Management Sciences, 33 (4), 261-86.
	
	\bibitem{Khan2016} 
	Khan M, Arshad M. (2016). UMVU estimation of reliability function and stress–strength reliability from proportional reversed hazard family based on lower records. American Journal of Mathematical and Management Sciences, 35 (2), 171-181.
	
	\bibitem{Kochar1987}
	Kochar S, Wiens D. (1987). Partial orderings of life distributions with respect to their aging properties. Naval Research Logistics, 34, 823-829.
	
	\bibitem{Kozubowski2016}
	Kozubowski TJ, Podg\'orski K. (2016). Transmuted distributions and random extrema. Statistics \& Probability Letters, 116, 6-8.
	
	\bibitem{Kumaraswamy}
	Kumaraswamy P. (1980). A generalized probability density function for double-bounded random processes. Journal of hydrology, 46 (1-2), 79-88.
	
	
	
	\bibitem{Lai2006}
	Lai C, Xie M. (2006). Stochastic Ageing and Dependence for Reliability. Springer.
	
	\bibitem{mahmoudi2012}
	Mahmoudi E, Jafari AA. (2012). Generalized exponential-power series distributions. Computational Statistics \& Data Analysis, 56 (12), 4047-4066.
	
	\bibitem{mahmoudi2017}
	Mahmoudi E, Jafari AA. (2017). The compound class of linear failure rate-power series distributions: Model, properties, and applications. Communications in Statistics-Simulation and Computation, 46 (2), 1414-1440.
	
	\bibitem{mahdavi2017}
	Mahdavi A, Kundu D. (2017). A new method for generating distributions with an application to exponential distribution. Communications in Statistics-Theory and Methods, 46 (13), 6543–6557.
	
	\bibitem{Maurya2020}
	Maurya SK, Nadarajah S. (2020). Poisson Generated Family of Distributions:A Review.Sankhya B. https://doi.org/10.1007/s13571-020-00237-8.
	
	\bibitem{Mahmoud2010}
	Mahmoud MA, Alam FM. (2010). The generalized linear exponential distribution. Statistics and Probability Letters, 80, 1005-1014.
	
	\bibitem{Marshall1997}
	Marshall, AW, Olkin I. (1997). A new method for adding a parameter to a family of distributions with application to the exponential and Weibull families. Biometrika, 84 (3), 641-652.
	
	\bibitem{Marshall2007}
	Marshall, AW, Olkin I. (2007). Life distributions: Structure of Nonparametric, Semiparametric, and Parametric Families. New York: Springer.
	
	
	\bibitem{Merovci2013} 
	Merovci F. (2013). Transmuted lindley distribution. International Journal of Open Problems in Computer Science and Mathematics, 6 (2), 63-72.
	
	\bibitem{Murthy2004} 
	Murthy D, Xie M, Jiang R. (2004). Weibull models. Wiley series in probability and statistics, Trenton, NJ, John Wiley and Sons.
	
	\bibitem{Nair2013}
	Nair NU, Sankaran PG, Balakrishnan N. (2013). Quantile-based reliability analysis. Boston, Birkhäuser.	
	
	\bibitem{Nofal2018} 
	Nofal ZM, Afify AZ, Yousof HM, Cordeiro GM. (2017). The generalized transmuted-G family of distributions, Communications in Statistics-Theory and Methods, 46 (8), 4119-4136.
	
	\bibitem{FA2020} 
	Peña-Ramírez FA, Guerra RR, Canterle DR, Cordeiro GM. (2020). The logistic Nadarajah–Haghighi distribution and its associated regression
	model for reliability applications. Reliability Engineering \& System Safety, 204, 107196.
	
		
	
	
	\bibitem{R}
	R Development Core Team (2021), R: A language and environment for statistical computing, R Foundation for Statistical Computing, Vienna, Austria.URL http://www.R-project.org/.
	
	\bibitem{Rahman2020}
	Rahman MM, Al-Zahrani B, Shahbaz S, Shahbaz MQ. (2020). Transmuted Probability Distributions: A Review. Pakistan Journal of Statistics and Operation Research, 16 (1), 83-94.
	
%
	 
	%
	
	\bibitem{Sarhan2013}
	Sarhan AM, Ahmad AEA, Alasbahi IA. (2013). Exponentiated generalized linear exponential distribution, Applied Mathematical Modelling, 37 (5), 2838-2849.
	
	\bibitem{Shakhatreh2021}
	Shakhatreh MK, Dey S, Alodat MT. (2021). Objective Bayesian analysis for the differential entropy of the Weibull distribution. Applied Mathematical Modelling, 89 (1), 314-332.
	
	\bibitem{Shanno1970}
	Shanno DF. (1970). Conditioning of quasi-Newton methods for function minimization. Mathematics of Computation, 24 (111), 647-656. 
	
	\bibitem{Shaw2007}
	Shaw WT, Buckley IR. (2007). The alchemy of probability distributions: Beyond gramcharlier \& cornish-fisher expansions, and skew-normal or kurtotic-normal distributions. http://arxiv.org/pdf/0901.0434v1.
	
	\bibitem{Shaw2009}
	Shaw WT, Buckley IR. (2009). The alchemy of probability distributions: Beyond GramCharlier expansions, and a skew-kurtotic-normal distribution from a rank transmutation map. arXiv:0901.0434.
	
	\bibitem{Singla2012}
	Singla N, Jain K, Sharma SK. (2012). The Beta Generalized Weibull distribution: Properties and applications. Reliability Engineering \& System Safety, 102, 5-15.
	
	\bibitem{Shakhatreh2021}
	Shakhatreh M, Dey S, Kumar D. (2021). Inverse Lindley power series distributions: A new compounding family and regression model with censored data. Journal of Appied Statistics (Accepted).
	
	\bibitem{Swain1988}
	Swain JJ, Venkatraman S, Wilson JR. (1988). Least-squares estimation of distribution functions in johnson's translation system. Journal of Statistical Computation and Simulation, 29 (4), 271-297.
	
	
	
	\bibitem{tahir2016}
	Tahir MH, Cordeiro GM. (2016). Compounding of distributions: a survey and new generalized classes. Journal of Statistical Distributions and Applications, 3 (13), 1-35.
	
	\bibitem{tahir2020}
	Tahir MH, Hussain MA, Cordeiro GM, El-Morshedy M., Eliwa MS. (2020). A new Kumaraswamy generalized	family of distributions with properties, applications, and bivariate extension. Mathematics, 8 (11), 1989. https://doi.org/10.3390/math8111989.
	
	\bibitem{Tanis2020}
	Tani\c{s} C, Sara\c{c}o\u{g}lu B. (2020). On the record-based transmuted model of Balakrishnan and He based on Weibull distribution. Communications in Statistics-Simulation and Computation. DOI: 10.1080/03610918.2020.1740261.
	
	\bibitem{Tian2014}
	Tian Y, Tian M, Zhu Q. (2014) Transmuted linear exponential distribution: A new generalization of the linear exponential distribution. Communications in Statistics-Simulation and Computation, 43 (10), 2661-77.
	
	\bibitem{Zeng2016}
	Zeng H, Lan T, Chen Q. (2016). Five and four-parameter lifetime distributions for bathtub-shaped failure rate using Perks mortality equation. Reliability Engineering \& System Safety, 152, 307-315.
	
	\bibitem{Zografos2009}
	Zografos K, Balakrishnan N. (2009). On families of beta-and generalized gammagenerated distributions and associated inference. Statistical Methodology, 6 (4), 344-362.
\end{thebibliography}
\end{document}